\documentclass[11pt]{article}

\usepackage{amssymb}
\usepackage{amsthm}
\usepackage{amsmath,amsfonts,amssymb}

\allowdisplaybreaks
\usepackage[mathscr]{eucal}
\usepackage{exscale}
\usepackage[square,sort,comma,numbers]{natbib}
\setcitestyle{authoryear,round}

\usepackage{bm}
\usepackage{eqlist}
\usepackage[dvipsnames]{color}
\usepackage[Lenny]{fncychap}
\usepackage{multirow}
\usepackage{graphicx}
\usepackage{subfigure}
\usepackage{algpseudocode}
\usepackage{algorithm}
\usepackage{caption}
\usepackage{hyperref}
\usepackage{framed}
\usepackage{color}
\usepackage{booktabs}
\usepackage{makecell}

\usepackage{framed}
\usepackage{algpseudocode}
\usepackage{indentfirst} 
\usepackage{longtable}
\usepackage{tabu}
\usepackage{tikz}
\usepackage{chngpage}

\usepackage[colorinlistoftodos]{todonotes}

\usepackage{ulem}
\usepackage{soul}

\usepackage[utf8]{inputenc} % this is needed for umlauts
\usepackage[english]{babel} % this is needed for umlauts
\usepackage[T1]{fontenc}    % this is needed for correct output of umlauts in pdf
\usepackage{amssymb,amsmath,amsfonts} % nice math rendering
\usepackage{braket} % needed for \Set
\usepackage{caption}
\usepackage{algorithm}

\hypersetup{
	colorlinks=true,
	linkcolor=blue,
	citecolor=blue,
	citebordercolor={1 1 0},
}

\newtheorem{theorem}{Theorem}[section]
\newtheorem{lemma}[theorem]{Lemma}

\newtheorem{definition}[theorem]{Definition}
\newtheorem{remark}[theorem]{Remark}

\font\bigbf=cmbx10 scaled \magstep3

\begin{document}

\title{\bigbf The Share-a-Ride Problem with mixed ride-hailing and logistic vehicles}

\author{Wen Ji$^{a}$
\quad 
Shenglin Liu$^{a}$
\quad
Ke Han$^{b}\thanks{Corresponding author, e-mail: kehan@swjtu.edu.cn;}$
\quad 
Yanfeng Li$^{b}$
\quad 
Tao Liu$^{a}$
\\\\
$^{a}$ \textit{\small School of Transportation and Logistics,}\\
\textit{\small Southwest Jiaotong University, Chengdu, Sichuan 611756, China}\\
$^{b}$ \textit{\small School of Economics and Management,}\\
\textit{\small Southwest Jiaotong University, Chengdu, Sichuan 610031, China}
}

\maketitle

\begin{abstract}

This study explores the potential of using ride-hailing vehicles (RVs) for integrated passenger and freight transport based on shared mobility. In this crowd-sourced mode, ride-hailing platforms can profit from parcel delivery services, and logistics companies can reduce operational costs by utilizing the capacities of RVs. The \underline{S}hare-\underline{a}-\underline{R}ide \underline{p}roblem with \underline{r}ide-hailing and \underline{l}ogistic vehicles (SARP-RL) determines the number of logistic vehicles (LVs) and the assignment of passenger/parcel requests to RVs and LVs, aiming at maximizing the total RV profits and minimizing logistic costs. An exact solution framework is proposed by (1) generating a feasible trip that serves a given set of requests at maximal profits; (2) generating all feasible trips for the entire set of passenger and parcel requests via an efficient enumeration method; and (3) finding all Pareto-optimal solutions of the bi-objective problem via an $\varepsilon$-constraint method. Not only is the proposed method exact, it also converts the NP-hard problem to a simple vehicle-trip matching problem. More importantly, the total computational time can be compressed to an arbitrary degree via straightforward parallelization. A case study of the Manhattan network demonstrates the solution characteristics of SARP-RL. The results indicate that: (i) Coordinating RV and LV operations to serve passenger and parcel requests (SARP-RL) can simultaneously reduce logistic costs and increase RV profits. (ii) Key factors influencing the performance of SARP-RL include the RV fleet size, spatial distribution of parcel requests, passenger/parcel request ratio, and unit price of transport service, which are quantitatively analyzed to offer managerial insights for real-world implementation.
\end{abstract}

\noindent {\it Keywords: Share-a-Ride problem; ride-hailing vehicles; trip-vehicle assignment; bi-objective optimization; route planning}

\section{Introduction}

In urban transport systems, passenger and freight services are typically operated independently. However, it has been realized that better utilization of vehicle and road capacity can be achieved via integrated passenger-good transport \citep{LKRV2014, MPC2019}. Since then, several researchers have investigated the synergistic potential of such integrated systems, involving ride-hailing vehicles \citep{LDLY2022, RJCYXZL2021, ZMY2022, BSN2018}, buses \citep{AAG2021, PFWJY2021, HG2023, MPS2023}, metros \citep{DYSZYG2022, DHYRZ2018}, trains \citep{SBBC2022, BBC2018} and private cars \citep{MG2022, VK2022}.

Ride-hailing vehicles (RVs) are suitable for providing intra-city door-to-door delivery services for both passengers and parcels for their high flexibility. They can deliver small parcels such as mails, documents, and takeaway meals while serving passengers. The economic drive behind such a crowd-sourced transport mode is the expectation that ride-hailing platforms can profit from parcel delivery services, and logistics companies can reduce operational costs. \cite{LKRV2014} are the first to investigate the potential benefits of combining passenger and parcel flows using RVs, by formulating the {\it share-a-ride problem} (SARP) and designing a greedy solution method. The SARP has a restriction that does not allow two passengers to be served simultaneously by one vehicle, to ensure high-quality passenger service. Subsequently, some variants of the SARP have been proposed and studied in the literature, including the SARP with stochastic travel times and stochastic delivery locations \citep{LKVR2016b}, the general SARP, which allows the vehicle to transport more than one passenger at the same time \citep{YPRSJ2018, BSN2018}, the SARP with flexible compartments, which allows vehicles to adjust their compartment sizes for parcel delivery. The models of SARP and its extensions developed in the above studies are all based on the well-known {\it dial-a-ride problem} (DARP) \citep{Cordeau2006}, which is an NP-hard problem \citep{BKS1998}. The computational complexity of the SARP motivated these studies to resort to heuristic algorithms for large-scale instances, including the adaptive large neighborhood search  \citep{LKVR2016a, LKVR2016b}, simulated annealing \citep{YPRSJ2018, YIRL2021} and matheuristic \citep{YMGIJ2023}. 

In the real world, both passenger and parcel requests are exogenously given and need to be served. Existing studies only focus on optimizing RV operations without considering their coordination with logistic vehicles (LVs) that are set out to meet all parcel requests. In other words, they did not explicitly address cost reduction on the logistics' part, rendering limited insights regarding the impact of this emerging mode on passenger and freight transport. Motivated by this, this study investigates the integrated RV and LV operations to serve passenger and parcel requests, coined the \underline{s}hare-\underline{a}-\underline{r}ide \underline{p}roblem with \underline{R}Vs and \underline{L}Vs (SARP-RL), in which passenger requests can only be served by RVs while parcel requests can be served by both RVs and LVs. This problem encompasses request selection, route planning, and vehicle-request matching, which is NP-hard. We propose an exact solution method based on a decomposition scheme that reduces the computational complexity by converting the SARP-RL into a simple vehicle-trip matching problem. Specific contributions of this work are as follows:

\begin{itemize}
    \item \textbf{Conception and model.} Unlike conventional SARP that only focuses on RV operations, we explore the coordination between RV and LV operations, with a balanced consideration of RV profits and logistic costs via an $\varepsilon$-constraint method to obtain all Pareto-optimal solutions. This model takes a holistic view of integrated passenger-freight transport based on shared mobility, and sheds light on its economic implication for the logistic industry. 
    
    \item \textbf{Solution method.} A problem decomposition scheme is devised by generating all trips based on an efficient enumeration method, converting the NP-hard SARP-RL (consisting of request selection, route planning, and vehicle-request matching) into a simple vehicle-trip matching problem. The latter can be solved as mixed integer linear programs efficiently using off-the-shelf solvers.
   
    \item \textbf{Practical insights.} 
A case study of the Manhattan network demonstrates the solution characteristics of SARP-RL. The results indicate that: (1) Coordinating RV and LV operations to serve passenger and parcel requests (SARP-RL) can simultaneously reduce logistic costs and increase RV profits. (2) Key factors influencing the performance of SARP-RL include the RV fleet size, spatial distribution of parcel requests, passenger/parcel request ratio, and unit price of transport service, which are quantitatively analyzed to offer managerial insights for real-world implementation.
\end{itemize}

The remainder of this paper is organized as follows. Section \ref{sec_RW} reviews some relevant works. The Share-a-Ride Problem with mixed ride-hailing and logistic vehicles is formally defined in Section \ref{sec_PS}. Section \ref{sec_SF} develops a solution framework of SARP-RL. A case study of our methods is presented in Section \ref{sec_CS}, followed by some concluding remarks in Section \ref{sec_CD}.

\section{Related work} \label{sec_RW}
    Work related to this study is divided into three parts: people and freight integration systems, ride-sharing systems, and share-a-ride problem.

\subsection{Integrated passenger and freight transport}
    Crowd-sourced delivery systems have created many new delivery solutions and have received some attention in the academic literature \citep{AGB2021}. Some studies focus on people and freight integration systems: \cite{LKRV2014} and \cite{ZMY2022} utilize taxis or shared autonomous vehicles to provide door-to-door services for both passengers and parcels. The vehicle with passengers can simultaneously serve some parcel requests. \cite{HC2023} integrate freight transport into public transport. They consider a setting where freight originates from and is transhipped at several public transit stops. \cite{HJGC2023} focuses on exploring the potential of modular vehicle concepts and consolidation to enhance the efficiency of urban freight and passenger transport. \cite{DYSZYG2022} investigates a joint optimization problem of carriage arrangement and flow control in a metro-based underground logistics system, in which passengers and freights are allowable to share each service train. In summary, the integration of people and freight has been explored in various vehicle types. For a more detailed overview of this topic, I suggest referring to \cite{CJN2023}.
    
    In this study, we aim to explore the impact of combining passenger and parcel flows using RVs on RV profits, the costs of logistics companies, and passenger service levels.

\subsection{Ride-sharing systems}
    
    Ride-sharing systems aim to bring together travelers with similar requests. These systems may provide significant societal and environmental benefits by reducing the number of vehicles used for personal travel and improving the utilization of available seat capacity \citep{AESW2012}. 
    
    One of the key challenges in ride-sharing systems is the matching problem between drivers and requests \citep{WY2019}. For example, \cite{AESW2011} considers the problem of matching drivers and requests in the dynamic setting. They develop optimization-based approaches that aim at minimizing the total system-wide vehicle miles incurred by system users, and their travel costs. A simulation study based on 2008 travel demand data from metropolitan Atlanta indicates that ride-sharing may help to decrease traffic congestion and thereby reduce system-wide travel times. \cite{FKW2021} build a stylized model of a circular road and compare the average waiting times of passengers under various matching mechanisms. They find that the on-demand matching mechanism could result in higher or lower efficiency than the conventional street-hailing mechanism, depending on the parameters of the system. \cite{VSRSR2018} addresses the minimum fleet problem in an on-demand shared transportation service. Using taxi trip data from NYC for one year, they found that a method with near-optimal service levels would allow a 30\% reduction in fleet size by comparison with current operations. \cite{QDO2008} proposed a bus-sharing system named MAST. In MAST, vehicles may deviate from a fixed path consisting of a few mandatory checkpoints to serve demand distributed within a proper service area. \cite{FC2023} propose a mixed integer linear programming model for a dial-a-ride problem with modular platooning. The results show that modular vehicle technology can save up to 52\% in vehicle travel cost, 41\% in passenger service time, and 29\% in total cost against existing on-demand mobility services in the scenarios tested.

    This study allows passenger-and-parcel sharing, and parcel-and-parcel sharing, to solve the matching problem between vehicles and requests.

\subsection{Share-a-ride problem}
    The Dial-a-Ride problem (DARP) consists of designing vehicle routes and schedules for several users who specify pick-up and drop-off requests between origins and destinations. The aim is to design a set of minimum-cost set of vehicle routes accommodating all requests under several side constraints \citep{CL2003, Cordeau2006}. A more comprehensive review of the DARP can be found in \cite{HSKLPT2018}. The Share-a-Ride problem (SARP) was initially proposed by \cite{LKRV2014} to expand on the original DARP. In SARP, people and parcels can share the same taxis and allow rejections for both people and freight requests. 
    
    Some studies have focused on the SARP: \cite{LKRV2014} present MILP formulations for the SARP, which can be directly solved exactly using GUROBI. Due to the complexity of the problem, GUROBI can solve only small instances. To solve large-scale problems, they propose a greedy solution approach, which starts from a given route for handling the passenger requests and inserts the parcel requests into this route.
    \cite{LKVR2016a} propose an adaptive large neighborhood search (ALNS) heuristic to address the SARP. Compared to the MIP solver, their heuristic is superior in both the solution times and the quality of the obtained solutions if the CPU time is limited. \cite{LKVR2016b} consider two stochastic variants of the SARP: one with stochastic travel times and one with stochastic delivery locations. They design an adaptive large neighborhood search heuristic to solve the problems. \cite{YPRSJ2018} introduces an extension of the SARP, called the general share-a-ride problem (G-SARP). Different from the traditional SARP, G-SARP allows the vehicle to transport more than one passenger at the same time. A simulated annealing (SA) algorithm is proposed to solve G-SARP. \cite{BSN2018} propose the share-a-ride with parcel lockers problem (SARPLP), in which both passenger and parcel requests are pooled in mixed-purpose compartmentalized SAVs. Different from the traditional SARP, they also allow the vehicle to transport more than one passenger at the same time. \cite{YIRL2021} presents the share-a-ride problem with flexible compartments (SARPFC). The SARPFC allows taxis to adjust their compartment size within the lower and upper bounds while maintaining the same total capacity permitting them to service more parcels while simultaneously serving at most one passenger. They propose a new variant of the Simulated Annealing algorithm called Simulated Annealing with Mutation Strategy to solve SARPFC. \cite{YMGIJ2023} develop a new matheuristic algorithm that combines the simulated annealing with mutation strategy and the set partitioning approach to improve the reported solutions of SARP benchmark instances. These studies, which are further summarized in Table \ref{tab_SARP_study}, all formalize the problem similar to traditional DARP \citep{Cordeau2006, MHBD2016, MBC2017} and employ heuristic algorithms to solve large-scale instances due to the size and complexity of the model. 

    This paper differs from the aforementioned studies in two ways: (a) We propose a new variant of SARP, called the SARP-RL. This problem investigates the integrated RV and LV operations to serve passenger and parcel requests, maximize RV profits, and minimize the costs of logistics companies. (b) We propose an exact solution framework based on a problem decomposition approach to solve the SARP-RL, which can be efficiently solved using off-the-shelf solvers. Compared to directly using solvers \citep{LKRV2014, BSN2018}, our methods can more efficiently solve the problem.
    
    \begin{table}[H]
		\centering
		\caption{Relevant literature on SARP and its extensions.}
		\small{
		\begin{tabular}{|m{0.10\textwidth}| m{0.39\textwidth} |m{0.12\textwidth} | m{0.28\textwidth}|}
		\hline
		Study & Model features & Method type & Solution method
		\\
		\hline
		\cite{LKRV2014}  & Allow passenger-parcel sharing and parcel-parcel sharing.& Exact (4-12 requests) & Gurobi for small instances and greedy method for large-scale instances.
		\\
		\hline
		\cite{LKVR2016a}   & Allow passenger-parcel sharing and parcel-parcel sharing.  & Heuristic & Adaptive large neighborhood search heuristic.  
        \\
        \hline
		\cite{LKVR2016b}   & Consider stochastic travel times and stochastic delivery locations & Heuristic & Adaptive large neighborhood search heuristic.   
        \\
        \hline
		\cite{YPRSJ2018}   & Allow passenger-passenger sharing, passenger-parcel sharing, and parcel-parcel sharing &Heuristic & Simulated annealing heuristic.   
        \\
        \hline
		\cite{BSN2018}   & Allow passenger-passenger sharing, passenger-parcel sharing, and parcel-parcel sharing &Exact (8-32 requests) & GUROBI.   
        \\
        \hline
		\cite{YIRL2021}  & (i) Allow passenger-parcel sharing, and parcel-parcel sharing; (ii) Allow vehicles to adjust their compartment size. & Heuristic & Simulated annealing heuristic with mutation strategy.   
        \\     
        \hline
		\cite{YMGIJ2023}   & -  & Heuristic & Matheuristic algorithm that combines the simulated annealing with mutation strategy and the set partitioning approach.   
        \\
        \hline
		This study  & (i) Allow passenger-parcel sharing and parcel-parcel sharing; (ii) Requires simultaneous consideration of assigning requests to two types of vehicles and ensuring that all parcel requests are completed; (iii) Bi-objective optimization & Exact (100 requests) & Problem decomposition approach $+$ $\varepsilon$-constraint method $+$ Mixed integer linear program (GUROBI).   
        \\
		\hline
		\end{tabular}
		}
		\label{tab_SARP_study}
    \end{table}
    
\section{Problem statement and preliminaries} \label{sec_PS}

\subsection{Relevant notions}

  Table \ref{tab_notations} lists key parameters and variables used in this paper.

    \setlength\LTleft{0pt}
    \setlength\LTright{0pt}
    \begin{longtable}{@{\extracolsep{\fill}}rl}
    \caption{Notations and symbols}
    \label{tab_notations} 
    \\
    \hline
    \multicolumn{2}{l}{Sets}     
        \\
        \hline
        $\mathcal{R}_{P}$  & Set of passenger requests;
        \\
        $\mathcal{R}_{F}$  & Set of parcel (freight) requests;
        \\
        $\mathcal{R}$  & Set of requests, $\mathcal{R} = \mathcal{R}_{P} \cup \mathcal{R}_{F}$;
        \\
        $\mathcal{T}_{P}$  & Set of trips with passenger requests only;
        \\
        $\mathcal{T}_{F}$  & Set of trips with parcel requests only;
        \\
        $\mathcal{T}_{M}$  & Set of trips with mixed passenger and parcel requests;
        \\
        $\mathcal{T}$  & Set of all trips, $\mathcal{T} = \mathcal{T}_{P} \cup \mathcal{T}_{F} \cup \mathcal{T}_{M}$; 
        \\
        $\hat{\mathcal{T}}$ & Set of feasible trips, as defined by \eqref{SARP_2}-\eqref{SARP_18};
        \\
        $V^{p,o}$  & Set of passenger origin stops, $V^{p,o}=\{1, 2, ...,n\}$;
        \\
        $V^{p,d}$   & Set of passenger destination stops, $V^{p,d}=\{1+\delta, 2+\delta, ...,n+\delta\}$;
        \\
        $V^{f,o}$  & Set of parcel origin stops, $V^{f,o}=\{n+1, n+2, ...,n+m\}$;
        \\
        $V^{f,d}$ & Set of parcel destination stops, $V^{f,d}=\{\delta+n+1, \delta+n+2, ...,\delta + n+m\}$;
        \\
        $V$  & Set of all vehicle stops, $V = V^{p,o} \cup V^{p,d} \cup V^{f,o} \cup V^{f,d}$;
        \\
        $\mathcal{K}_{R}$  & Set of RVs, $\mathcal{K}_{R}=\{1,2,...|\mathcal{K}_{R}|\}$;
        \\
        $\mathcal{K}_{L}$ & Set of LVs, $\mathcal{K}_{L}=\{1, 2, ..., |\mathcal{K}_{L}|\}$;
           \\
            \hline
    \multicolumn{2}{l}{Parameters and constants}    
   \\  \hline
            $\phi_{rp}$ & Binary parameter that equals if trip $p$ contains request $r$;
        \\
        $Q_{k}$ & Capacity of vehicle $k$, $k \in \mathcal{K}_{R} \cup \mathcal{K}_{L}$;
        \\
        $q_{i}$  & Vehicle load at stop $i$;  
        \\
        $\text{TD}(i,j)$  & Travel distance between stops $i$ and $j$;
        \\
        $\text{TT}(i,j)$  & Travel time between stops $i$ and $j$;
	\\
        $\sigma$ & Maximum wait time between request submission and the start of service;
        \\
        $\Delta_{P}$ & Maximum delay tolerable for passenger requests;
        \\
        $\Delta_{F}$ & Maximum delay tolerable for parcel requests;
        \\
        $t_{r}^{s}$ & Submission time of the request $r\in\mathcal{R}$;
        \\
        $t_{r}^{*}$ & Earliest possible arrival time at destination for request $r\in\mathcal{R}$;
        \\
        $\eta$  & Maximum number of stops during one passenger service trip;
        \\
        $\alpha$  & Fixed income associated with serving a passenger request $r\in\mathcal{R}_P$;
        \\
        $\beta$  & Fixed income associated with serving a parcel request $r\in\mathcal{R}_F$;
        \\
        $\gamma_{1}$  & Variable income per kilometer for serving a passenger request $r\in\mathcal{R}_P$;
        \\
        $\gamma_{2}$  & Variable income per kilometer for serving a parcel request $r\in\mathcal{R}_F$;
        \\
        $\gamma_{3}$  & Average cost per kilometer for vehicle operations (energy, personnel, etc.);
        \\
        $\gamma_{4}$  & Penalty per minute of travel delay for passenger requests;
        \\
        $\xi_{p}$ &  Profit from serving trip $p$;
        \\
        \hline
        \multicolumn{2}{l}{Auxiliary variables}                
        \\
        \hline
        $\tau_{i}$  & Arrival time at stop $i$;
        \\	
        $w_{i}$  & Vehicle load after visiting stop $i$;
        \\	
        $P_{i}$  & The index of stop $i$ in a visit sequence of the vehicle route;
        \\	 
        \hline
        \multicolumn{2}{l}{Decision variables}                 
        \\
        \hline
        $x_{i,j}$  & Binary variable that equals $1$ if vehicle travels directly from stop $i$ to stop $j$;
        \\
        $y_{p}$ & Binary variable that equals $1$ if trip $p$ is selected.
        \\
        \hline
    \end{longtable}

{\bf Travel distance and time.} Given a road network, the travel distance between origin $o$ and destination $d$ is denoted by $\text{TD}(o,d)$, which is equal to the length of the shortest path. In this study, we assume the vehicle speed is constant, denoted by $v$. Then, we can calculate the travel time between $o$ and $d$, denoted $\text{TT}(o,d)$.

{\bf Requests.} A request $r$ is defined as a tuple $(o_{r}, d_{r}, t_{r}^{s}, t_{r}^{pl}, t_{r}^{p}, t_{r}^{d}, t_{r}^{*})$, including its origin $o_{r}$, destination $d_{r}$, submission time $t_{r}^{s}$, the latest acceptable pick-up time $t_{r}^{pl} = t_{r}^{s} + \sigma$ where $\sigma$ the maximum wait time, the pick-up time $t_{r}^{p}$, the drop-off time $t_{r}^{d}$, and the earliest possible time at which the destination could be reached $t_{r}^{*} = t_{r}^{s} + \text{TT}(o_{r}, d_{r})$. In this study, there are two types of requests: passenger and parcel requests. Both requests may encounter drop-off delays, which is the difference between the actual drop-off time and the earliest possible delivery time $t_{r}^{d} - t_{r}^{*}$. This work sets maximum delays tolerated by passenger requests $\Delta_{P}$ and parcel requests $\Delta_{F}$.

{\bf Trips.} A trip $p = \{r_{1}, ..., r_{|p|}\}$ is a set of requests to be served by the same vehicle. Note that each request $r_i$ can refer to both passenger and parcel.

{\bf Vehicles.} There are two types of vehicles, ride-hailing vehicles (RVs) and logistic vehicles (LVs). The RVs can serve both passenger and parcel requests, separately and simultaneously, while the LVs can only serve parcel requests. Each vehicle $k$ has a capacity $Q_{k}$.

\subsection{Problem description}
    
   \begin{definition}\label{probdef}
   \textbf{The share-a-ride problem with ride-hailing vehicles and logistic vehicles (SARP-RL)} Given a set of passenger requests $\mathcal{R}_{P}=\{r_{1}, ..., r_{n}\}$, a set of parcel requests $\mathcal{R}_{F}=\{r_{n+1}, ..., r_{n+m}\}$, and a set of RVs $\mathcal{K}_{R} = \{1, ..., |\mathcal{K}_{R}|\}$. The SARP-RL aims to determine (1) the minimum LV fleet size, and (2) the optimal assignment of requests to two types of vehicles, to maximize the total RV profits and minimize logistic costs (by minimizing the fleet size), subject to a set of constraints $\mathcal{Z}$ pertinent to the level of service.
    \end{definition}

    In particular, this study considers the following set of constraints $\mathcal{Z}$:

    \begin{enumerate}
        \item By coordinating RVs and LVs, all parcel requests must be completed in full.
        \item RVs can serve both passenger and parcel requests, while LVs can only serve parcel requests.
        \item We allow passenger-parcel and parcel-parcel sharing, not passenger-passenger sharing to maintain the service level. 
        \item When a request $r \in \mathcal{R}_{P} \cup \mathcal{R}_{F}$ is submitted, its wait time is bounded by $\sigma$: $t_r^p\in[t_{r}^{s},\, t_{r}^{s} + \sigma]$.
        \item For each passenger request $r \in \mathcal{R}_{P}$, its maximum travel delay $t_{r}^{d} - t_{r}^{*} \leq \Delta_{P}$; For each parcel request $r \in \mathcal{R}_{F}$, its maximum travel delay $t_{r}^{d} - t_{r}^{*} \leq \Delta_{F}$.
        \item For each vehicle $k$, the load must not exceed the capacity $Q_{k}$ at any time.
        \item For each passenger request $r \in \mathcal{R}_{P}$, the maximum number of stops tolerable during one passenger service trip is $\eta$.
    \end{enumerate}

\subsection{The drivers' profits}
    This study adopts a similar cost structure as in  \cite{LKRV2014} and \cite{YMGIJ2023}. For RVs, their income associated with passenger request $r$ is
$$
\alpha + \gamma_1 \text{TD}(o_r,d_r) -\gamma_4 (t_r^d-t_r^*)
$$
\noindent where $\alpha$ is the fixed income, and $\gamma_1\text{TD}(o_r,d_r)$ is the variable income proportional to the travel distance, $\gamma_4(t_r^d-t_r^*)$ is the penalty associated with travel delays (note that such delays must be bounded by $\Delta_P$). Similarly, the income associated with parcel request $r$ reads
$$
\beta+ \gamma_2\text{TD}(o_r, d_r)
$$
\noindent where no penalties are imposed for delayed drop-offs (note that such delays must be bounded by $\Delta_F$). Then, for an RV serving a trip $T$, its total profit is expressed as
\begin{equation}\label{profV}
\Phi_T\doteq \sum_{r\in T\cap \mathcal{R}_P}\alpha + \gamma_1 \text{TD}(o_r,d_r) -\gamma_4 (t_r^d-t_r^*) + \sum_{r\in T\cap\mathcal{R}_F} \beta+ \gamma_2\text{TD}(o_r, d_r) - \gamma_3 L_T
\end{equation}
 \noindent where $L_T$ is the total distance traveled to complete the trip $T$, and $\gamma_3$ is the cost per unit distance. Note that \eqref{profV} equally applies to LVs, because in this case $T\cap \mathcal{R}_p=\emptyset$.
 
\section{Model development and solution algorithm} \label{sec_SF}
    In this section, we mathematically articulate the SARP-RL problem from Definition \ref{probdef}, while developing a solution framework based on a decomposition scheme.

\subsection{Overview}
   Given a set of passenger requests $\mathcal{R}_P$, a set of parcel requests $\mathcal{R}_F$ and a set of RVs $\mathcal{K}_R$. The exact solution framework of SARP-RL computes all Pareto-optimal solutions and consists of the following steps, shown in Figure \ref{fig_overview}.

\begin{itemize}
\item[(1)] \textbf{Optimal trip for given requests}: This step checks whether a set of requests can form a trip, constrained by $\mathcal{Z}$. If so, it outputs the feasible route of the trip in the network to serve these requests that maximizes the profit. The process of feasibility check or route optimization is formulated as a mixed integer program (MILP).

\item[(2)] \textbf{Generate all trips}: This step generates all feasible trips using an efficient enumeration method, which drastically reduces the required instances of feasibility check or route optimization (MILP) in Step (1). In addition, an empirical analysis of the enumeration results suggest that the size of those MILPs are rather small, meaning the overall computational burden is negligible.

\item[(3)] \textbf{Trip-vehicle assignment}: Building on Steps (1) and (2), the SARP-RL is converted to a trip-vehicle assignment problem. This step introduces an $\varepsilon$-constraint method framework to maximize RV driver profits and minimize the LV fleet size, or to find all Pareto-optimal solutions.

\end{itemize}
    \begin{figure}[H]
       \centering
       \includegraphics[width=0.8\textwidth]{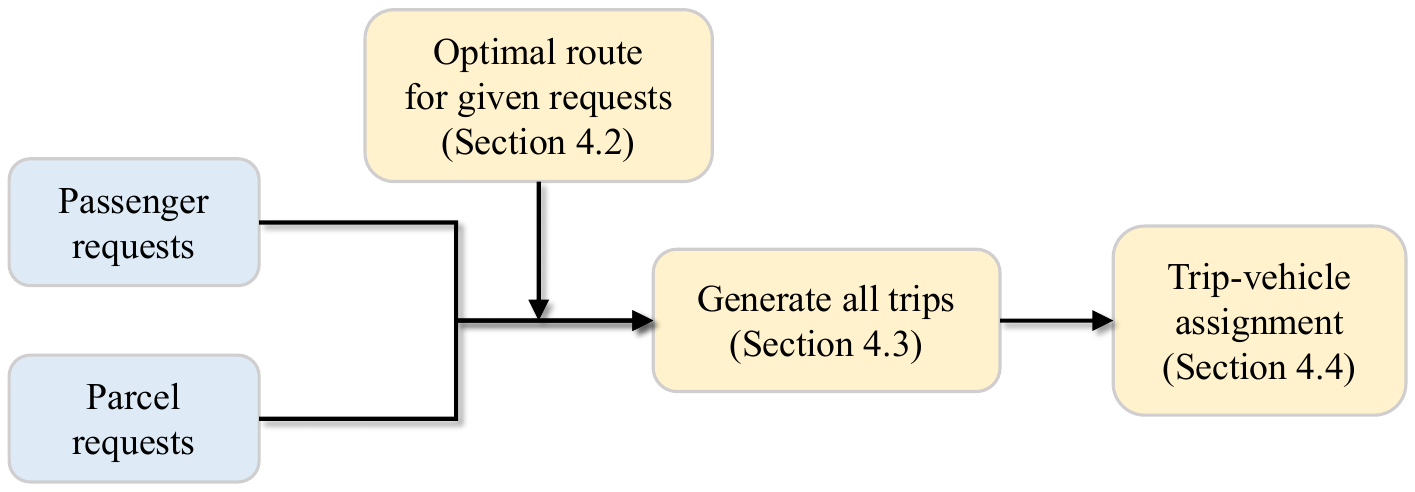}
       \caption{Solution framework for the proposed SARP-RL.}
       \label{fig_overview}
    \end{figure}

\subsection{Optimal route for given requests}
    
    Given a subset of requests $\mathcal{R}^{*} \subseteq \mathcal{R}$, where $\mathcal{R}^{*} = \mathcal{R}_{P}^{*} \cup \mathcal{R}_{F}^{*}$, corresponding to subsets of passenger and parcel requests, respectively. This section aims to check whether $\mathcal{R}^{*}$ can be served by the same vehicle and, if so, obtain the optimal route that maximizes the profit.

    Let $\delta$ be the total number of requests in $\mathcal{R}^{*}$, including $n$ passenger requests and $m$ parcel requests. Let $V^{p,o}=\{1,2, ..., n\}$ be set of passenger origins, $V^{p,d} = \{1+\delta, 2+\delta, ..., n+\delta\}$ be set of passenger destinations, $V^{f,o} = \{n+1, n+2, ..., n+m\}$ be set of parcel origins, and $V^{f,d} = \{n+1+\delta, n+2+\delta, ..., n+m+\delta\}$ be set of parcel destinations. Let $V = V^{p,o} \cup V^{p,d} \cup V^{f,o} \cup V^{f,d}$ be all the vehicle stops. We introduce $0$ and $2 \delta + 1$ to represent the virtual origin and destination depots. Each stop $i \in V$ is associated with a load $q_{i}$ ($q_{0} = q_{2\delta + 1} = 0$), and $\forall i \in  V^{p,o} \cup V^{f,o}$, $q_{i} = -q_{i + \delta}$. Let $x_{i,j} = 1$, if the vehicle travels directly from stop $i$ to stop $j$. For each stop $i\in V$, let $\tau_{i}$ be the time of arrival at stop $i$, and $w_{i}$ be the load of the vehicle after visiting stop $i$. Finally, we use variables $P_{i} \in \{1, 2, ..., 2(\delta +1)\}$ to define the index of stop $i$ in a visit sequence of the vehicle route, and maximum number of stops that can be accepted between one passenger request is defined by $\eta$. The problem of finding the optimal route for given requests $R^{*}$ is formulated as the following mixed integer linear program:
    
     \begin{multline}
        \label{SARP_1}
        \max_{x_{i,j}} ~ \left[\sum_{i \in V^{p,o}}(\alpha + \gamma_{1}\text{TD}(i, i+\delta))
        - \gamma_{4}\sum_{i \in V^{p,o}}(\tau_{\delta+i} - t_{i}^{*})\right] 
        \\
        + \sum_{i \in V^{c,o}}\left[\beta + \gamma_{2}\text{TD}(i, i+\delta)\right]
        - \gamma_{3}\sum_{i \in V}\sum_{j \in V}\text{TD}(i,j)x_{i,j}
    \end{multline}
    
    \begin{eqnarray}
        \label{SARP_2}
        \sum_{i \in V}x_{0,i} = 1 
        \\
        \label{SARP_3}
        \sum_{i \in V}x_{i,2\delta + 1} = 1 
        \\
        \label{SARP_4}
        \sum_{j \in V \cup \{2\delta + 1\}}x_{i,j} = \sum_{j \in V \cup \{0\}}x_{j,i}  & & \forall i \in V
        \\
        \label{SARP_5}
        \sum_{j \in V \cup \{2\delta + 1\}}x_{i,j} = 1 & & \forall i \in V
        \\
        \label{SARP_6}
        \tau_{j} - \tau_{i} \geq \text{TT}(i,j) + M_1(x_{i,j}-1)  & & \forall i,j \in V \cup \{0, 2\delta + 1\}
        \\
        \label{SARP_7}
        t_{i}^{s} \leq \tau_{i} \leq t_{i}^{s} + \sigma  & & \forall i \in V^{p,o} \cup V^{f,o}
        \\
        \label{SARP_8}
        \tau_{\delta + i} - t_{i}^{*} \leq \Delta_{P}  & & \forall i \in V^{p,o}
        \\
        \label{SARP_9}
        \tau_{\delta + i} - t_{i}^{*} \leq \Delta_{F}  & & \forall i \in V^{f,o}
        \\
        \label{SARP_10}
        \tau_{i} \leq \tau_{\delta + i}  & & \forall i \in V^{p,o} \cup V^{f,o}
        \\
        \label{SARP_11}
        w_{j} - w_{i}  \geq q_{j} + M_2(x_{i,j}-1)  & & \forall i,j \in V \cup \{0, 2\delta + 1\}
        \\
        \label{SARP_12}
        \max \{0, q_{i}\} \leq w_{i} \leq \min \{Q_{k}, Q_{k} + q_{i}\}  & & \forall i \in V \cup \{0, 2\delta + 1\}
        \\
        \label{SARP_13}
        P_{i} + 1 - P_{j} \geq M_3(x_{i,j} - 1)  & & \forall i,j \in V
        \\
        \label{SARP_14}
        P_{i} + 1 - P_{j} \leq M_3(1- x_{i,j})  & & \forall i,j \in V
        \\ 
        \label{SARP_15}
        P_{j+\delta} - P_{j} - 1 \leq \eta  & & \forall j \in V^{p,o}
        \\
        \label{SARP_16}
        0 \leq P_{i} \leq 2\delta + 1  & & \forall i \in V
        \\
        \label{SARP_17}
        x_{i,j} \in \{0,1\}  & & \forall i,j \in V \cup \{0, 2\delta+1\}
        \\
        \label{SARP_18}
        \tau_{i}, w_{i}, P_{i} \in \mathbb{R}_{+}  & & \forall i \in V \cup \{0, 2\delta+1\}
    \end{eqnarray}
    
    The objective \eqref{SARP_1} is to maximize the driver profits. Constraints \eqref{SARP_2} and \eqref{SARP_3} guarantee that the route starts at the origin depot and ends at the destination depot. Constraint \eqref{SARP_4} ensures that every stop except the origin and the destination must have one preceding and one succeeding stop. Constraint \eqref{SARP_5} ensures that each request must be served. Constraint \eqref{SARP_6} computes the visit time of stop $i$. $M_1$ is a sufficiently large number. The time window of origin stops constraints are defined in \eqref{SARP_7}. Constraint \eqref{SARP_8} and \eqref{SARP_9} ensure that passenger and parcel requests cannot exceed the maximum travel delay tolerated. Constraint \eqref{SARP_10} ensures that for each request, the pick-up time must be earlier than the drop-off time. Constraint \eqref{SARP_11} computes the loads of the vehicle after visiting stop $i$. $M_2$ is a sufficiently large number. Since it is not allowed two passenger requests served by the same vehicle, the parameter value is necessary to meet $Q_{k} < 2q_{i}$ $(\forall i \in V_{p,o})$. Vehicle capacities are defined using constraint \eqref{SARP_12}. constraints \eqref{SARP_13} and \eqref{SARP_14} define the service sequence of the stops. $M_3$ is a sufficiently large number. Constraint \eqref{SARP_15} guarantees that passenger service has a higher priority: we can insert at most $\eta$ stops between the pickup and drop-off point of a passenger request. Constraints \eqref{SARP_16} - \eqref{SARP_18} are the decision variable constraints.

Note that the constraints \eqref{SARP_2}-\eqref{SARP_18} may not admit feasible solutions and, in that case, the given requests $\mathcal{R}^*$ cannot form a trip. Otherwise, this subproblem outputs the optimal route in the network that maximizes the profit while serving all trips in $\mathcal{R}^*$.

\subsection{Generate all feasible trips} \label{sec_GT}
    A trip is defined as a subset of requests $\mathcal{R}^{*}\subset\mathcal{R}$ that can be served by the same vehicle, whose feasibility is checked by \eqref{SARP_2}-\eqref{SARP_18}. A single request may be included in several feasible trips of varying sizes. Our strategy for generating all feasible trips is to progressively augment existing trips, subject to the feasibility check provided by \eqref{SARP_2}-\eqref{SARP_18}. Such a process is facilitated by Lemma \ref{Lemma_1}, which considerably reduces the search space. 
    
    \begin{lemma}
       If a subset of requests $\mathcal{R}^*=\{r_1, ..., r_n\}$ cannot form a trip, then neither can $\mathcal{R}^*\cup\{r_{n+1}\}$ for any $r_{n+1}\in \mathcal{R}\setminus\mathcal{R}^*$.
        \label{Lemma_1}
    \end{lemma}
    \begin{proof}
	If $\exists r_{n+1}\in \mathcal{R}\setminus\mathcal{R}^*$ such that $\mathcal{R}^*\cup\{r_{n+1}\}$ can form a trip, then $\mathcal{R}^*$ itself must be a feasible trip, leading to contradiction.
    \end{proof}

Inspired by \cite{ASWFR2017}, we begin with trips that consist of a single request and progressively augment those trips until they become infeasible. Let $\mathcal{T}(l)$ be the set of all feasible trips consisting of $l$ requests, this preliminary process is summarized in Algorithm \ref{alg_GAT}.

 \begin{algorithm}[h!]
         \caption{Generate all trips \citep{ASWFR2017}}
         \begin{tabbing}
         \hspace{0.01 in}\=  \hspace{0.9 in}\= \kill % set up two tab positions
         \>{\bf Input}  \> Ordered set of all requests $\mathcal{R}=\{r_1,\ldots,r_{|\mathcal{R}|}\}$;
         \\
         \> {\bf Initialize} \> Set of singleton trips $\mathcal{T}(1)\doteq \{p=\{r_i\}:~1\leq i\leq |\mathcal{R}| \}$; $l = 1$.
        \\
         \>  {\bf Step 1} \>  Generate trip set $\mathcal{T}(l + 1)$ by checking the feasibility of $p\cup \{r_j\}, \forall p\in\mathcal{T}(l)$,
         \\
         \> \> $\forall r_j \notin p$;
         \\
         \>  {\bf Step 2} \>  If $\mathcal{T}(l + 1) \neq \emptyset$,
         let $l=l+1$, and go to Step 1; otherwise, terminate;
         \\
        \>{\bf Output}      \> Set of all trips $\mathcal{T}(1),\, \mathcal{T}(2),\,\ldots, \mathcal{T}(l)$.
        \end{tabbing}
         \label{alg_GAT}
    \end{algorithm}

A potentially time-consuming part in Algorithm \ref{alg_GAT} is Step 1, which needs to enumerate all possible candidate trips and check their feasibility by solving problem \eqref{SARP_1}-\eqref{SARP_18}. To further reduce the complexity of this procedure, we note that only a partial enumeration would suffice, as asserted in Lemma \ref{Lemma_2} and executed in Algorithm \ref{alg_GNT}.

    \begin{lemma}\label{Lemma_2}
	Let $\mathcal{R}=\{r_1, \ldots, r_{|\mathcal{R}|}\}$ be an ordered set of all the requests. Let $\mathcal{T}(l)$ be the set of feasible trips consisting of $l$ requests. For each trip $p\in\mathcal{T}(l)$, let 
$$
j^*_p=\underset{1\leq j\leq |\mathcal{R}|}{\text{argmax}}\,\{r_j\in p\}$$ 
\noindent be the highest request index in $p$. Then, 
\begin{equation}\label{lemmaeqn1}
\mathcal{T}(l+1)=\bigcup_{p\in\mathcal{T}(l)} \big\{p\cup r_{j}:~j_p^*+1\leq j \leq |\mathcal{R}| \big\}\cap \hat{\mathcal{T}}
\end{equation}
    \end{lemma}

\begin{proof}
We denote the set on the RHS of \eqref{lemmaeqn1} to be $A$. Clearly, $\mathcal{T}(l+1)$ can be expressed via exhaustive enumeration:
$$
\mathcal{T}(l+1)=\bigcup_{p\in\mathcal{T}(l)} \big\{p\cup r_{j}:~1\leq j \leq |\mathcal{R}|,  r_{j} \notin p \big\}\cap \hat{\mathcal{T}}
$$
Therefore, it suffices to show that $\forall p\in\mathcal{T}(l)$, and $\forall j<j_p^*$, $p\cup r_j \cap \hat{\mathcal{T}}\in A$.  Indeed, if $p\cup r_j$ is feasible, then $q\doteq p\cup r_j\setminus r_{j_p^*}$ is also feasible per Lemma \ref{Lemma_1}, which means $q\in \mathcal{T}(l)$. Then, $p\cup r_j= q\cup r_{j_p^*}$ for some $q\in\mathcal{T}(l)$ (and $j_q^*+1\leq j_p^*$), which means $p\cup r_j \cap\hat{\mathcal{T}}\in A$. Otherwise, if $p\cup r_j$ is infeasible, then $p\cup r_j \cap\hat{\mathcal{T}}=\emptyset \in A$. This finishes the proof.
\end{proof}

    \begin{algorithm}[h!]
         \caption{Generate $\mathcal{T}(l+1)$ based on $\mathcal{T}(l)$}
         \begin{tabbing}
         \hspace{0.01 in}\=  \hspace{0.9 in}\= \kill % set up two tab positions
         \>{\bf Input}  \> Trip set $\mathcal{T}(l)$; ordered set of all requests $\mathcal{R}=\{r_1, \ldots, r_{|\mathcal{R}|}\}$.
         \\
         \> {\bf Initialize} \> $\mathcal{T}(l+1) = \emptyset$.
        \\
         \>  {1:} \>  For each trip $p \in \mathcal{T}(l)$:
         \\
         \> {2:}\> \hspace{0.2 in} Find the highest request index $j^*_p$ in $p$;
         \\
         \> {3:}\>  \hspace{0.2 in} For request $j \in \{j^*_p + 1, \ldots, |\mathcal{R}|\}$:
         \\
        \> {4:}\>   \hspace{0.4 in} Apply \eqref{SARP_2}-\eqref{SARP_18} to check whether $p\cup\{r_j\}$ can form a feasible trip;
         \\
        \> {5:}\>   \hspace{0.4 in} If $p\cup\{r_j\}$ is feasible, add it to $\mathcal{T}(l+1)$;
         \\
         \> {6:}\>  \hspace{0.2 in} End For
         \\
         \> {7:}\> End For
         \\
        \>{\bf Output}      \> Trip set $\mathcal{T}(l+1)$.
        \end{tabbing}
         \label{alg_GNT}
    \end{algorithm}

By replacing Step 1 in Algorithm \ref{alg_GAT} with Algorithm \ref{alg_GNT}, one could considerably reduce the number of candidate trips for the feasibility check. To see this, we use several datasets from the numerical case study and evaluate the number of candidate trips to be checked for feasibility, as shown in Table \ref{tab_trip_combinations}. Compared to direct enumeration, the technique introduced by \cite{ASWFR2017} (Algorithm \ref{alg_GAT}) can considerably reduce the number of candidate trips. Moreover, the method proposed in Algorithm \ref{alg_GNT} can further reduce such number by a factor of 7 to 11.

In addition, Figure \ref{fig_trips} shows the number of candidate trips as $l$ varies. Most trips in these datasets contain between 3 to 6 requests, with a maximum of 9 requests. Therefore, the computational burden of \eqref{SARP_1}-\eqref{SARP_18} is insignificant and can be quickly solved using off-the-shelf solvers. In fact, as the last column of Table \ref{tab_trip_combinations} shows, the average CPU time for solving such a problem is $3\times 10^{-3}$ s.

	\begin{table}[h!]
	\centering
	\caption{Number of trip evaluations required by different enumeration methods. 8 request datasets are randomly generated from the Manhattan case study. The last column is the total CPU time for checking all feasible trips based on Algorithm \ref{alg_GNT}.}
	\label{tab_trip_combinations}
	\resizebox{\columnwidth}{!}{
	\begin{tabular}{|c|c|c|cc|c|c|}
	\hline
	\multirow{2}{*}{\begin{tabular}[c]{@{}c@{}}Seed \\ \#\end{tabular}} & \multirow{2}{*}{\begin{tabular}[c]{@{}c@{}}Num. of \\ requests\end{tabular}}  & \multirow{2}{*}{\begin{tabular}[c]{@{}c@{}}Max \\ value of $l$\end{tabular}} & \multicolumn{4}{c|}{Number of candidate trips to be evaluated for feasibility} \\ \cline{4-7} 
	  &     &          & \multicolumn{1}{c|}{Direct enumeration}  &Alg. \ref{alg_GAT}     & With Alg. \ref{alg_GNT} & CPU time (s)\\ \hline
	0 & 100 &  9 & \multicolumn{1}{c|}{$\sum_{i=1}^{10} {100\choose i}\approx1.94\times 10^{13}$} &6,942,396 & 651,399   &  1954 \\ \hline
	1 & 100 &  8 & \multicolumn{1}{c|}{$\sum_{i=1}^{9} {100\choose i}\approx 2.11\times 10^{12}$} & 8,460,320  & 945,330   &  2836 \\ \hline
	2 & 100 &  7 & \multicolumn{1}{c|}{$\sum_{i=1}^{8} {100\choose i}\approx 2.03\times 10^{11}$} &  5,114,247  & 549,368  &  1648  \\ \hline
	3 & 100 &  8 & \multicolumn{1}{c|}{$\sum_{i=1}^{9} {100\choose i}\approx 2.11\times 10^{12}$} & 6,973,555 & 885,274   &  2656 \\ \hline
	4 & 100 &  9 & \multicolumn{1}{c|}{$\sum_{i=1}^{10} {100\choose i}\approx 1.94\times 10^{13}$}&  7,014,696 & 928,855   &  2787 \\ \hline
	5 & 100 &  8 & \multicolumn{1}{c|}{$\sum_{i=1}^{9} {100\choose i}\approx 2.11\times 10^{12}$} & 7,034,300 & 965,966  &  2898  \\ \hline
	6 & 100 &  8 & \multicolumn{1}{c|}{$\sum_{i=1}^{9} {100\choose i}\approx 2.11\times 10^{12}$} & 5,006,316 & 509,737   & 1529  \\ \hline
	7 & 100 &  8 & \multicolumn{1}{c|}{$\sum_{i=1}^{9} {100\choose i}\approx 2.11\times 10^{12}$} & 6,935,121 & 783,873   &  2352 \\ \hline
	\end{tabular}
	}
	\end{table}

    \begin{figure}[H]
       \centering
       \includegraphics[width=0.7\textwidth]{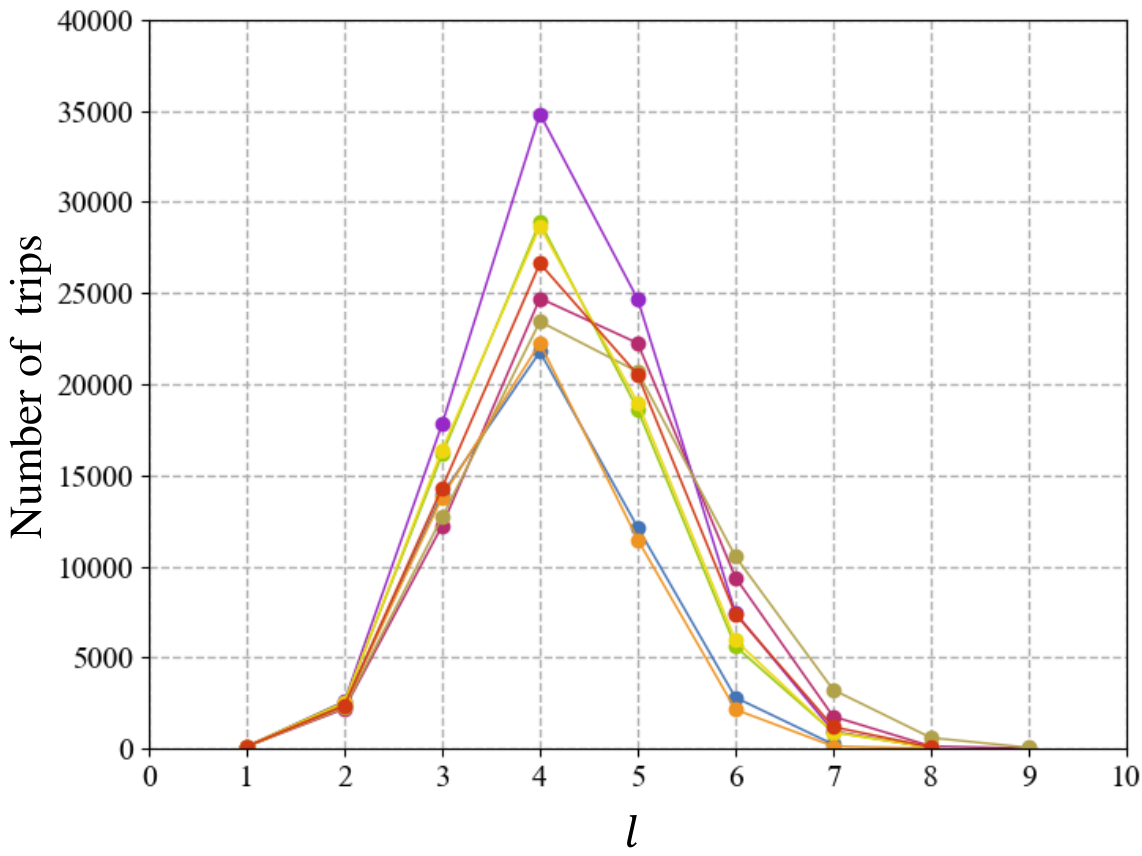}
       \caption{Number of trips of 8 random datasets under various request quantities $l$.}
       \label{fig_trips}
   \end{figure}

\subsection{Trip-vehicle assignment}
	Based on Section \ref{sec_GT}, we have generated all the feasible trips, each representing a tour for a vehicle, thereby converting the SARP-RL into a trip-vehicle assignment problem with much lower complexity. Regarding the optimization objective, both RV drivers' profits and LVs' fleet size (representing logistic cost) are within the purview of such an assignment problem. In this section, we introduce an $\varepsilon$-constraint solution framework to investigate the trade-off between these two objectives. Specifically, we set the  LVs' fleet size as a new constraint and transform the bi-objective problem into a single-objective optimization problem. The procedure is as follows. 
\begin{enumerate}
\item First, the initial value of $\varepsilon$ is set to the minimum fleet size $N_{L}^{\text{min}}$ required to complete all parcel requests without any RVs, which is obtained by solving the sub-problem in Section \ref{subsecMinLVsize}. 

\item Then, a sub-routing is employed to fully utilize all the LVs with a fleet size $|\mathcal{K}_{L}|=\varepsilon$, by solving the sub-problem in Section \ref{subsecMaxLVprofit}. 

\item Given that all parcel requests must be served, maximize the profit of RV drivers $\Phi_{RV}^{\varepsilon}$ by solving the sub-problem in Section \ref{subsecMRVprofit}. A tuple $(\varepsilon,\, \Phi_{RV}^{\varepsilon})$ is obtained.

\item Repeat Steps 2-3 above with smaller $\varepsilon$ to get $(\varepsilon,\, \Phi_{RV}^{\varepsilon})$, until the problem in Step 3 becomes infeasible (the parcel requests cannot be fully served if the LVs' fleet size is insufficient).

\end{enumerate}
	
The overall solution procedure is outlined in Algorithm \ref{alg_SF}, with individual sub-routines elaborated in Sections \ref{subsecMinLVsize}-\ref{subsecMRVprofit}.

    \begin{algorithm}[H]
         \caption{($\varepsilon$-constraint method)}
         \begin{tabbing}
         \hspace{0.01 in}\=  \hspace{0.9 in}\= \kill % set up two tab positions
         \>{\bf Input}  \> Passenger-only trips $\mathcal{T}_{P}$; parcel-only trips $\mathcal{T}_{F}$; passenger-parcel-mixed trips $\mathcal{T}_{M}$; 
         \\
         \> \> RV fleet size $|\mathcal{K}_{R}|$.
         \\
         \> {\bf Initialize} \> Solution set $S=\emptyset$; Set the candidate trip set for RVs $\mathcal{T}_{RV} = \emptyset$;
        \\
         \>   \>  Solve problem \eqref{Min_DV_FS_1}-\eqref{Min_DV_FS_3} to obtain the minimum LV fleet size $N_{L}^{\text{min}}$; 
         \\
         \>   \>  Set $\varepsilon =N_{L}^{\text{min}}$;
         \\
         \>  {\bf Step 1} \> Solve problem \eqref{Max_DV_Profit_1}-\eqref{Max_DV_Profit_4} with LVs' fleet size $|\mathcal{K}_{L}|=\varepsilon$;  
         \\
         \> \> The set of parcel requests served by the LVs is denoted $\mathcal{R}_{F}^{'}$;
         \\
         \>  {\bf Step 2} \> Update the candidate trips for RVs: $\mathcal{T}_{RV}=\mathcal{T}_{P} \cup \big\{p\in \mathcal{T}_F\cup\mathcal{T}_M:~p\cap\mathcal{R}_F^{'}=\emptyset \big\}$;
         \\
         \>  {\bf Step 3} \> Solve problem \eqref{Max_RV_Profit_1}-\eqref{Max_RV_Profit_5}  based on candidate trip set $\mathcal{T}_{RV}$. If the problem is 
         \\
         \> \> feasible with optimal objective value $\Phi_{RV}^{\varepsilon}$, add $(\varepsilon, \Phi_{RV}^{\varepsilon})$ to $S$, let $\varepsilon = \varepsilon - 1$ 
         \\
         \> \> and go to \textbf{Step 1}. If the problem is infeasible, terminate the algorithm;
         \\
         \> {\bf Step 4} \> Find all non-dominated solutions in $S$;
         \\
        \>{\bf Output}      \> All Pareto-optimal solutions.
        \end{tabbing}
         \label{alg_SF}
    \end{algorithm}

\subsubsection{Minimum LV fleet size for parcel requests}\label{subsecMinLVsize}

    This section addresses the minimum fleet size $N_{L}^{\text{min}}$ required to complete all parcel requests by an LV fleet. Let $\mathcal{T}_{F}$ be the set of trips with parcel requests only. The binary parameter $\phi_{rp}=1$ if trip $p$ contains request $r$. Let $y_{p} = 1$, if trip $p$ is selected. The formulation is as follows:

    \begin{eqnarray}
        \label{Min_DV_FS_1}
        \min_{y_{p}}  \sum_{p \in \mathcal{T}_{F}} y_{p}
    \end{eqnarray}
    
    \begin{eqnarray}
        \label{Min_DV_FS_2}
        \sum_{p \in \mathcal{T}_{F}}\phi_{rp}y_{p} = 1  & & \forall r \in \mathcal{R}_{F}
        \\
        \label{Min_DV_FS_3}
        y_{p} \in \{0, 1\}  & & \forall p \in \mathcal{T}_{F}
    \end{eqnarray}

\noindent The objective \eqref{Min_DV_FS_1} is to minimize the LV fleet size required to complete all parcel requests. Constraints \eqref{Min_DV_FS_2} ensure that each parcel request is served. Constraints \eqref{Min_DV_FS_3} are the decision variable constraints.

\subsubsection{Maximizing the utility of LVs} \label{subsecMaxLVprofit}
    In practical applications, the company will prioritize ensuring that all LVs are fully utilized. This section addresses the problem of maximizing the utility of LVs when the LV fleet size $|\mathcal{K}_{L}|=\varepsilon$. Let $\mathcal{T}_{F}$ be the set of trips composed of parcel requests only. The binary parameter $\phi_{rp}=1$ if trip $p$ contains request $r$. Parameter $\xi_{p}$ represents the profit that the vehicle can obtain from trip $p$, which is calculated by the Model Eqs. \eqref{SARP_1}-\eqref{SARP_18}. Let $y_{p} = 1$, if trip $p$ is selected. The formulation is as follows:

    \begin{eqnarray}
        \label{Max_DV_Profit_1}
        \max_{y_{p}}  \sum_{p \in \mathcal{T}_{F}} \xi_{p}y_{p}
    \end{eqnarray}
    
    \begin{eqnarray}
        \label{Max_DV_Profit_2}
        \sum_{p \in \mathcal{T}_{F}}\phi_{rp}y_{p} \leq 1  & & \forall r \in \mathcal{R}_{F}
        \\
        \label{Max_DV_Profit_3}
        \sum_{p \in \mathcal{T}_{F}}y_{p} \leq \varepsilon 
        \\
        \label{Max_DV_Profit_4}
        y_{p} \in \{0, 1\}  & & \forall p \in \mathcal{T}_{F}
    \end{eqnarray}

    The objective \eqref{Max_DV_Profit_1} is to maximize the utility of LVs. Constraints \eqref{Max_DV_Profit_2} ensure that each parcel request can be served at most once. Constraint \eqref{Max_DV_Profit_3} ensures that the total number of LVs cannot exceed $\varepsilon$. Constraints \eqref{Max_DV_Profit_4} are the decision variable constraints.

\subsubsection{Maximizing the total RV driver profits}\label{subsecMRVprofit}
    This section addresses the problem of maximizing the total RV driver profits when the RV fleet size is $|\mathcal{K}_{R}|$. $\mathcal{T}_{RV}$ is the candidate trip set of RVs. The set of parcel requests that are completed by LVs is marked as $\mathcal{R}_{F}^{'}$. In the trip-RV assignment problem, it is necessary to ensure that all parcel requests are completed. The binary parameter $\phi_{rp}=1$ if trip $p$ contains request $r$. Parameter $\xi_{p}$ represents the profit that the vehicle can obtain from trip $p$. Let $y_{p} = 1$, if trip $p$ is selected. The formulation is as follows:

    \begin{eqnarray}
        \label{Max_RV_Profit_1}
        \max_{y_{p}} ~ \sum_{p \in \mathcal{T}_{RV}} \xi_{p}y_{p}
    \end{eqnarray}
    \begin{eqnarray}
        \label{Max_RV_Profit_2}
        \sum_{p \in \mathcal{T}_{RV}}\phi_{rp}y_{p} \leq 1  & & \forall r \in \mathcal{R}_{P}
        \\
        \label{Max_RV_Profit_3}
        \sum_{p \in \mathcal{T}_{RV}}\phi_{rp}y_{p} = 1  & & \forall r \in \mathcal{R}_{F} \setminus \mathcal{R}_{F}^{'}
        \\
        \label{Max_RV_Profit_4}
        \sum_{p \in \mathcal{T}_{RV}}y_{p} \leq |\mathcal{K}_{R}|  
        \\
        \label{Max_RV_Profit_5}
        y_{p} \in \{0, 1\}  & & \forall p \in \mathcal{T}_{RV}
    \end{eqnarray}

    The objective \eqref{Max_RV_Profit_1} is to maximize the total RV driver profits. Constraints \eqref{Max_RV_Profit_2} ensure that each passenger request can be served at most once. Constraints \eqref{Max_RV_Profit_3} ensure that all parcel requests are served. Constraint \eqref{Max_RV_Profit_4} ensures that the total number of RVs cannot exceed $|K_{R}|$. Constraints \eqref{Max_RV_Profit_5} are the decision variable constraints.

\section{Case study} \label{sec_CS}
In this section, we consider a case study in Manhattan, New York, to illustrate the impact of the shared transport scheme (SARP-RL) as well as the performance of the proposed solution method. All the computational performances reported below are based on a Microsoft Windows 10 platform with Intel Core i9 - 3.60GHz and 16 GB RAM, using Python 3.8 and Gurobi 9.1.2.

\subsection{Experimental settings}

\subsubsection{Passenger request data}

    The passenger requests are generated based on the Manhattan taxi dataset from \cite{NYCdata2024}. The NYC taxi dataset contains several relevant entries, including submission time, origin, and destination. In the original dataset, the origins and destinations of passenger requests are represented as taxi zones, as shown in Figure \ref{fig_manhattan} (left). In the experiment, we randomly select street intersections within the taxi zone as the origin or destination of the request in the road network. Figure \ref{fig_manhattan} (right) shows the spatial distribution of the number of trip requests in Manhattan between 13:00 and 14:00 on 3 January 2022. In this study, passenger requests are randomly selected during this period.

    \begin{figure}[H]
       \centering
       \includegraphics[width=1.0\textwidth]{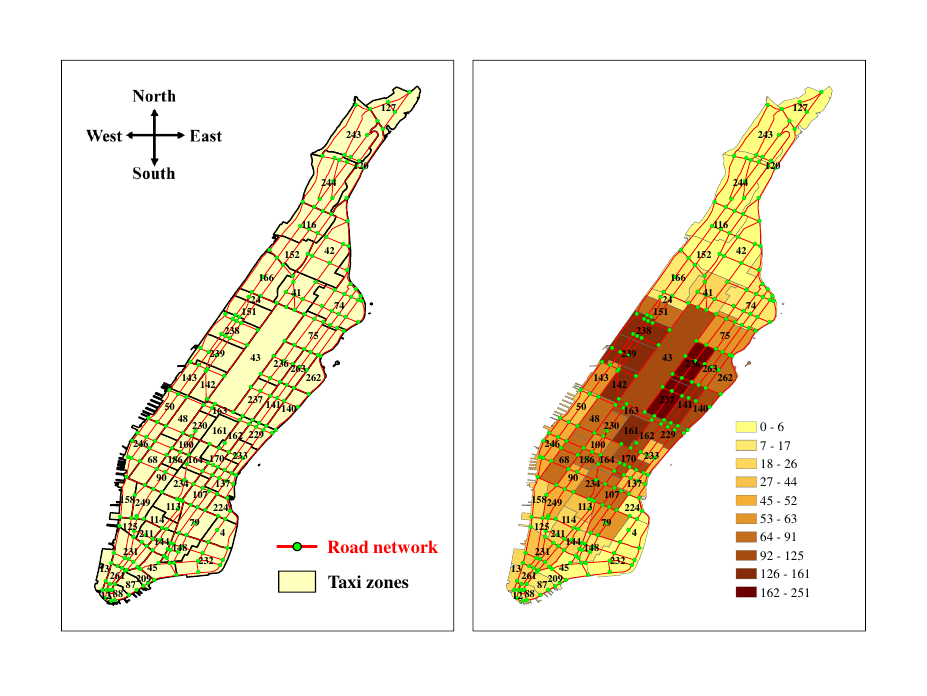}
       \caption{Left: Road network and taxi zones of Manhattan, New York City. Right: Spatial distribution of number of passenger trip requests in Manhattan between 13:00 and 14:00 on January 3, 2022.}
       \label{fig_manhattan}
    \end{figure}

\subsubsection{Parcel request data}

    Inspired by \cite{LKRV2014}, the generation of parcel requests considers various spatial distributions to represent a diversified portfolio of freight demand scenarios.
    
      \begin{itemize}
        \item \textbf{Scatter-Scatter (SS)}: Both origins and destinations are scattered, corresponding to the pattern of personal shipping requests. 
        \item \textbf{Scatter-Cluster (SC)}: The origins are scattered, and destinations are clustered, corresponding to the pattern of first-mile transport to a logistic center.
         \item \textbf{Cluster-Scatter (CS)}: The origins are clustered, and destinations are scattered, corresponding to the pattern of last-mile transport from a logistic center. 
    \end{itemize}
    
 Furthermore, in the SC or CS patterns, we consider different locations of the clustered destinations/origins. For example, SC (South) means the clustered destinations are located in the south of the road network, and CS (North) means the clustered origins are located in the north of the network. Figure \ref{fig_parcel_spatial_distribution} visualizes these different parcel flow patterns. 
    
    \begin{figure}[H]
       \centering
       \includegraphics[width=\textwidth]{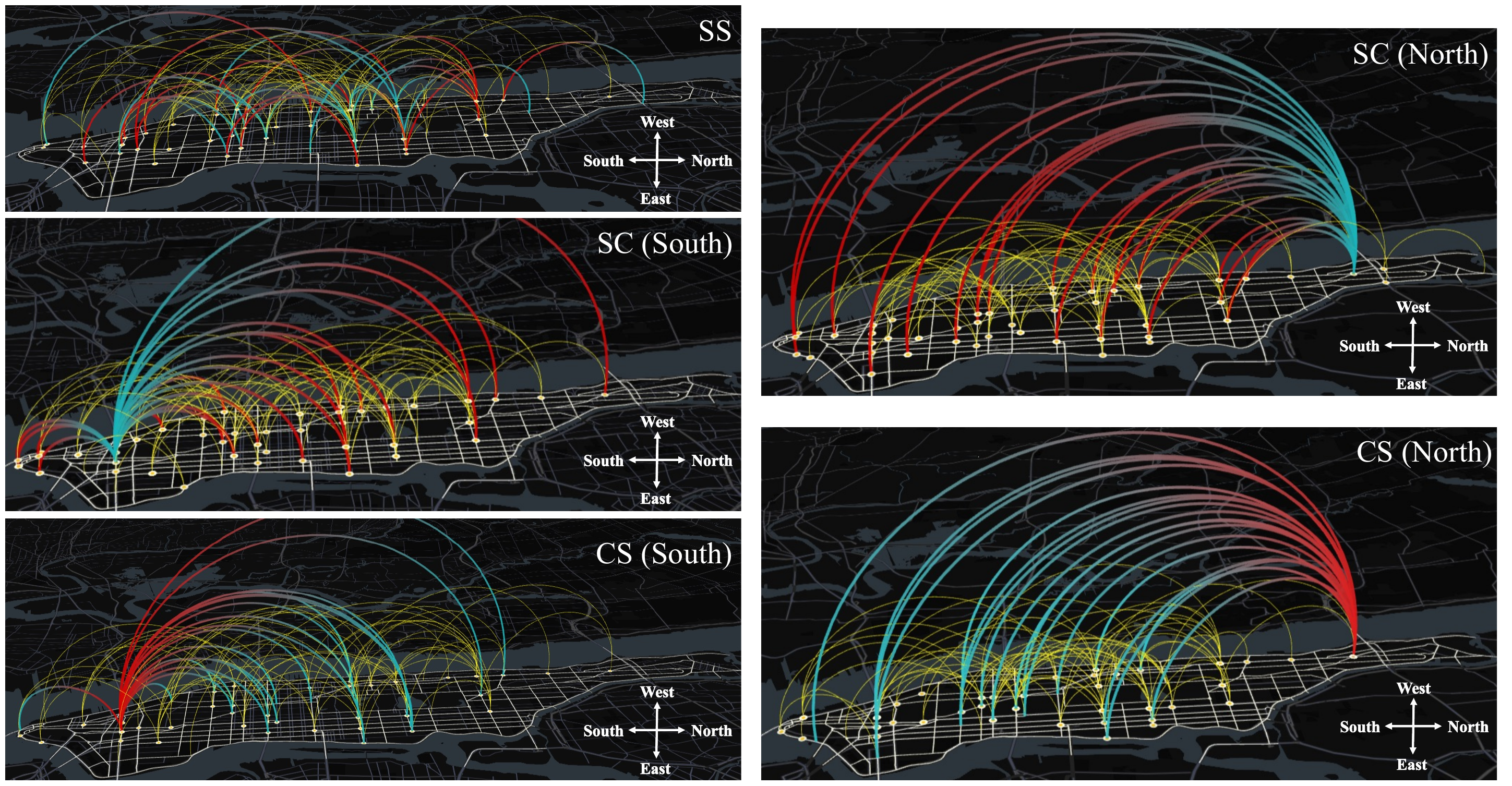}
       \caption{Five spatial distributions of parcel requests: SS, SC (South), SC (North), CS (South) and  CS (North). The red ends of the arcs represent the origins, and the blue ends represent the destinations. The yellow arcs represent passenger requests.}
       \label{fig_parcel_spatial_distribution}
   \end{figure}

\subsubsection{Demand configuration and model parameters}

Given the passenger and parcel demand distributions, we generate the mixed passenger-parcel demand as follows. Given the demand scenario (SS, SC, CS, etc.), we randomly generate $X$ number of passenger requests and $Y$ number of parcel requests, using random seed \#$Z$. For example, SS-76-24-0 means demand scenario SS with 76 passenger and 24 parcel requests generated using random seed \#0. All the results reported below for each demand case are averaged over 8 independent runs (\#0-7) with different seeds.

The parameters used in this study are given in Table \ref{tab_parameter_values}.

\begin{table}[h]
\centering
\caption{Parameters used for the SARP-RL.}
\label{tab_parameter_values}
\resizebox{\columnwidth}{!}{%
\begin{tabular}{|c|c|}
\hline
Parameters                                                            & Values \\ \hline
Vehicle speed $v$                                                     & 30 (km/h)     \\ \hline
Capacity of vehicle $Q_{k}$                                                & 6      \\ \hline
Load of passenger request $q_{i}$, $\forall i \in \mathcal{R}_{P}$          & 4      \\ \hline
Load of parcel request $q_{i}$, $\forall i \in \mathcal{R}_{F}$            & 1      \\ \hline
Maximum number of stops during one passenger service trip $\eta$ & 2 \\ \hline
Maximum wait time between request submission and start of service $\sigma$ & 5 (min)     \\ \hline
Maximum delay tolerated for passenger requests $\Delta_{P}$              & 10 (min)    \\ \hline
Maximum delay tolerated for parcel requests $\Delta_{F}$                 & 15 (min)    \\ \hline
Fixed income associated with serving a passenger request $\alpha$                 & 5      \\ \hline
Fixed income associated with serving a parcel request $\beta$                     & 3      \\ \hline
Variable income per kilometer for serving a passenger request $\gamma_1$    & 2.4    \\ \hline
Variable income per kilometer for serving a parcel request $\gamma_2$      & 1.2    \\ \hline
Average cost per kilometer for vehicle operations (energy, personnel, etc.) $\gamma_3$  & 0.6    \\ \hline
Penalty per minute of travel delay for passenger requests $\gamma_4$      & 0.5  \\ \hline
Number of requests   & 100  \\ \hline
\end{tabular}%
}
\end{table}

\subsection{Evaluation of the optimization results}

We consider the following three benchmark scenarios to demonstrate the potential of the SARP-RL and the effectiveness of the $\varepsilon$-constraint method framework. In this study, we use the fleet size of LVs to represent logistics costs. The reason is as stated in Remark \ref{logistic_cost}.

 \begin{remark}
\label{logistic_cost}
Logistics costs consist of three parts: the fixed costs of the LVs, the operating costs of the LVs, and the costs paid to the RVs. The operating costs of the LVs mainly depend on the sum of the distances between the origin and destination of all the parcel requests served by the LVs. The costs paid to RVs are calculated based on \eqref{profV} and also depend mainly on the sum of the travel distances between the origin and destination of all parcel requests served by RVs. According to the definition of SARP-RL, all parcel requests must be completed under the coordination of RVs and LVs. Therefore, the total distance between the origin and destination of all parcel requests served by RVs and LVs is certain, which means that the sum of the operating costs of LVs and the costs paid to RVs is basically unchanged in different solutions. Therefore, in this study, we can replace logistics costs with the fleet size of LVs.
 \end{remark}

\begin{enumerate}
    \item[(1)] \textbf{RV-only:} The maximum RV profits of  serving passenger requests only.
    \item[(2)] \textbf{LV-only:} The minimum LV fleet size required to complete all parcel requests.
    \item[(3)] \textbf{SARP:} Existing studies on SARP \citep{LKRV2014, YPRSJ2018} focus on optimizing RV operations without considering their coordination with the LVs. In this benchmark scenario, we first maximize RV profits based on passenger and parcel requests. Then, for the remaining parcel requests, we solve a minimum fleet size problem for the LVs to serve them. 
\end{enumerate}

Table \ref{tab_pareto_RV_LV} presents the Pareto-optimal solutions of the SARP-RL tested on the demand scenario SS-76-24, as well as comparison with the three benchmarks. The results indicate that: 
\begin{itemize}
\item[(1)] SARP-RL can simultaneously reduce the number of LVs and increase RV profits compared to RV-only and LV-only. When the RV number is small (e.g. 5), SARP-RL has an insignificant impact on RV profits (with some mixed results), although its performance on LV fleet size reduction is strong. When the number of RVs increases (above 10, for example), the SARP-RL excels in terms of both LV fleet size (in fact, no LVs are needed) and RV profits (mostly over 15\% increase; see Figure \ref{fig_RV_increase} Left), compared to the RV-only and LV-only modes. 

\item[(2)] In most cases, SARP-RL can reduce logistics costs (LV fleet size) compared to SARP (see Figure \ref{SARP_different_parcel_distribution} Left), by coordinating RV and LV services. However, the RV profits in the SARP mode are higher than those in the SARP-RL (see Figure \ref{SARP_different_parcel_distribution} Right), because of its profit maximizing nature. 

\end{itemize}

\begin{table}[h!]
\centering
\caption{The Pareto-optimal solutions of the SARP-RL and the comparison results with three benchmarks based on the demand scenario SS-76-24. Each 2-tuple $(\varepsilon, \Phi_{RV})$ in the Pareto-optimal solutions represents the fleet size of LVs and the total RV profits, respectively.}
\label{tab_pareto_RV_LV}
\resizebox{\columnwidth}{!}{%
\begin{tabular}{|c|c|c|c|c|c|c|}
\hline

  \multirow{3}{*}{\begin{tabular}[c]{@{}c@{}}Number\\ of RVs\end{tabular}} &
	\multirow{3}{*}{Seed no} &
  \multirow{3}{*}{\begin{tabular}[c]{@{}c@{}}SARP-RL Pareto-optimal \\ solutions $(\varepsilon,\,\Phi_{RV}^{\varepsilon})$\end{tabular}} &
  \multirow{3}{*}{\begin{tabular}[c]{@{}c@{}}LV-only\\ fleet size\end{tabular}} &
   \multirow{3}{*}{\begin{tabular}[c]{@{}c@{}}RV-only\\ profits\end{tabular}} & 
\multirow{3}{*}{\begin{tabular}[c]{@{}c@{}}SARP \\ fleet size\end{tabular}} & 
  \multirow{3}{*}{\begin{tabular}[c]{@{}c@{}}SARP \\ RV profit \end{tabular}} \\
              &    &         & &                 &   &                           \\ 
                &    &         & &                        &     &                    \\
              \hline
\multirow{12}{*}{5} & \#0  & (3, 332); (4, 340) & 8 & 317 & 5 & 401     \\ \cline{2-7} 
 &
  \#1 &
  \begin{tabular}[c]{@{}c@{}}(1, 312); (2, 340); \\ (3, 348); (6, 361)
  \end{tabular} &
  6 &
  361 & 4 & 
443  \\ \cline{2-7} 
 & \#2 &\begin{tabular}[c]{@{}c@{}}(2, 334); (4, 336); (5, 350)\end{tabular} &7 &339 & 6&408               \\ \cline{2-7} 
 & \#3  & (2, 344); (3, 362) & 7 & 341  &5 & 422   \\\cline{2-7} 
 & \#4  & (2, 328); (3, 331) & 6 & 328  &4 & 423
\\\cline{2-7} 
 & \#5 & \begin{tabular}[c]{@{}c@{}}(1, 299); (4, 313);\\ (5, 329); (6, 343)\end{tabular} & 6& 343 & 4&411 \\ \cline{2-7} 
 & \#6 &\begin{tabular}[c]{@{}c@{}}(2, 243); (3, 277);\\ (4, 301);(5, 303);\\ (6, 315); (7, 319)\end{tabular} &7&319& 5&
371 \\ \cline{2-7} 
 &
  \#7 &
  \begin{tabular}[c]{@{}c@{}}(2, 261); (3, 319); \\ (4, 323); (6, 328)\end{tabular} &
  7 &
  328 & 5&
406
  \\ 
  \hline
\multirow{8}{*}{10} & \#0 & (0, 710)              & 8 & 579  &5 & 727            \\ \cline{2-7} 
 & \#1 & (0, 746)              & 6 & 654  & 4  & 777             \\ \cline{2-7} 
 & \#2 & (0, 707)              & 7 & 624  & 3 & 740            \\ \cline{2-7} 
 & \#3 & (0, 724)              & 7 & 619  & 3 & 754            \\ \cline{2-7} 
 & \#4 & (0, 737)              & 6 & 616 &4  & 749            \\ \cline{2-7} 
 & \#5 & (0, 679)              & 6 & 598 &3  & 712            \\ \cline{2-7} 
 & \#6 & (0, 651)              & 7 & 573  &4 & 680            \\ \cline{2-7} 
 & \#7 & (0, 696)              & 7 & 592 &4  & 717           \\ \hline
\multirow{8}{*}{15} & \#0 & (0, 985)              & 8 & 813  & 3  & 989            \\ \cline{2-7}
 & \#1 & (0, 1023)             & 6 & 878  & 2 & 1032           \\ \cline{2-7}
 & \#2 & (0, 992)              & 7 & 853 & 2  & 991           \\ \cline{2-7}
 & \#3 & (0, 1002)             & 7 & 849   & 2   & 1010            \\ \cline{2-7}
 & \#4 & (0, 1003)             & 6 & 840  & 2 & 1005          \\ \cline{2-7}
 & \#5 & (0, 939)              & 6 & 812  & 3 & 948            \\ \cline{2-7}
 & \#6 & (0, 925)              & 7 & 783 & 2  & 931           \\ \cline{2-7}
 & \#7 & (0, 963)              & 7 & 805 &3  & 963           \\ \hline
\multirow{8}{*}{20} & \# 0 & (0, 1182)             & 8 & 981 & 0 & 1182            \\ \cline{2-7}
 & \#1 & (0, 1216)             & 6 & 1048 &1  & 1219            \\ \cline{2-7}
 & \#2 & (0, 1163)             & 7 & 1005 &1 & 1163           \\ \cline{2-7}
 & \#3 & (0, 1202)             & 7 & 1028 & 2 & 1204            \\ \cline{2-7}
 & \#4 & (0, 1196)             & 6 & 1006 &2 & 1196             \\ \cline{2-7}
 & \#5 & (0, 1133)             & 6 & 974  & 3 & 1135             \\ \cline{2-7}
 & \#6 & (0, 1120)             & 7 & 973  & 1 & 1121           \\\cline{2-7}
 & \#7 & (0, 1143)             & 7 & 964  & 2 & 1143          \\ \hline
\end{tabular}%
}
\end{table}

    \begin{figure}[h!]
       \centering
       \includegraphics[width=1.0\textwidth]{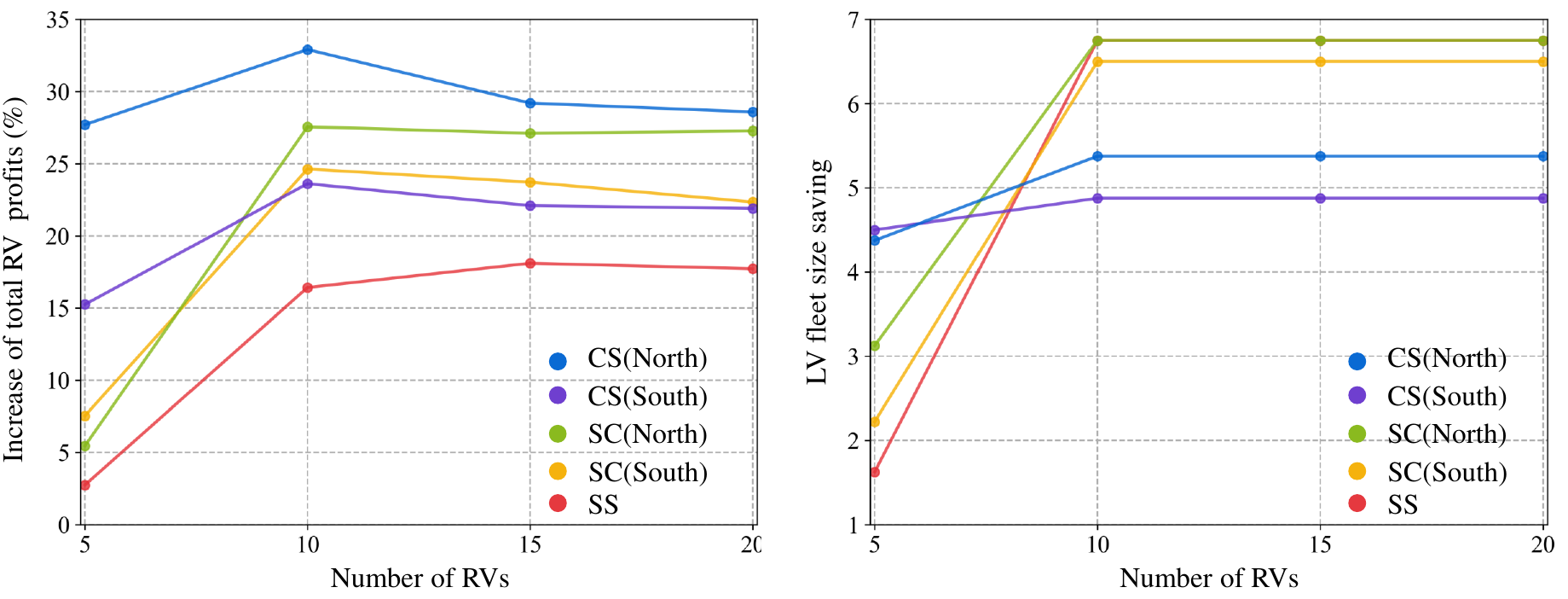}
       \caption{Left: Increase of total RV profits (\%) of five parcel distributions under the same number of RVs for SARP-RL versus RV-only. Right: LV fleet size saving of five parcel distributions under the same number of RVs for SARP-RL versus LV-only. The experimental results are based on the demand scenarios SS-76-24, SC(South)-76-24, SC(North)-76-24, CS(South)-76-24 and CS(North)-76-24.}
       \label{fig_RV_increase}
   \end{figure}

    \begin{figure}[h!]
       \centering
       \includegraphics[width=1.0\textwidth]{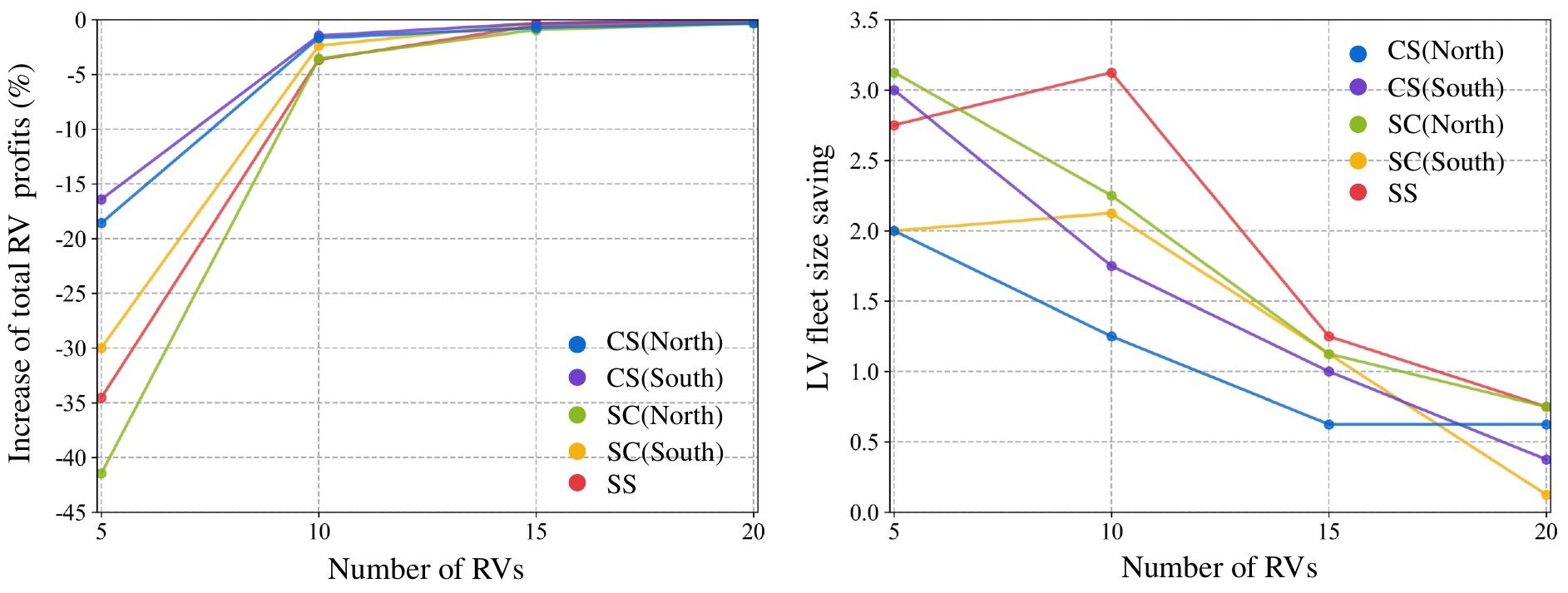}
       \caption{Left: Increase of total RV profits (\%) of five parcel distributions under the same number of RVs for SARP-RL versus SARP. Right: LV fleet size saving of five parcel distributions under the same number of RVs for SARP-RL versus SARP. The experimental results are based on the demand scenarios SS-76-24, SC(South)-76-24, SC(North)-76-24, CS(South)-76-24 and CS(North)-76-24.}
       \label{SARP_different_parcel_distribution}
   \end{figure}

	In terms of average profit per RV, Figure \ref{fig_per_vehicle} shows that:(1) The average profit per RV decreases as the number of RVs increases because online RVs tend to be saturated; (2) The SARP-RL in terms of average profit per RV (mostly over 15\% - 30\% increase), compared to the RV-only.
    \begin{figure}[h!]
       \centering
       \includegraphics[width=.6\textwidth]{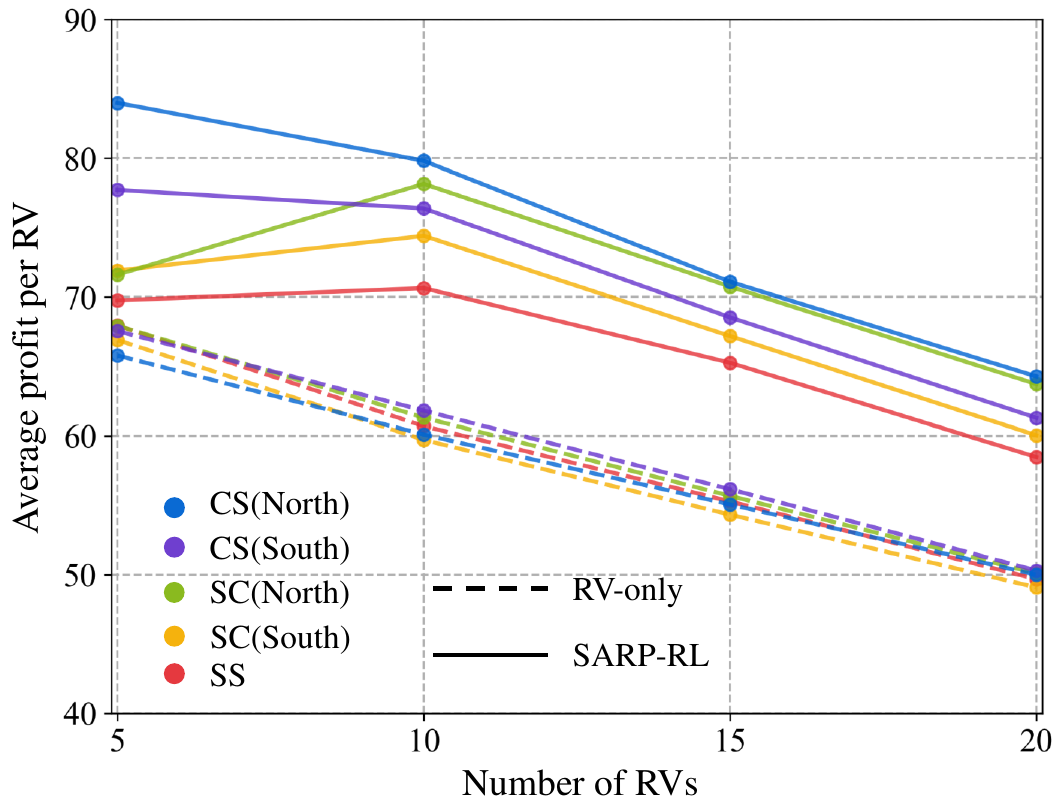}
       \caption{Average profit per vehicle of five parcel distributions under the same number of RVs. The experimental results are based on the demand scenarios SS-76-24, SC(South)-76-24, SC(North)-76-24, CS(South)-76-24 and CS(North)-76-24.}
       \label{fig_per_vehicle}
   \end{figure}

    Regarding the acceptance rate of passenger requests, Figure \ref{fig_acceptance_rate} shows that the SARP-RL will lead to a decrease in passenger acceptance rate compared to the RV-only mode. By comparing the results of different parcel distributions, the results indicate that: (1) if the cluster point of the origin or destination of the parcel requests is located in remote areas, this has a significant impact on the acceptance rate of the passenger requests, for example, SC(North) and CS(North) decrease by about 10\%-20\%. (2) if the distribution of parcel requests and passenger requests is similar, the impact on the acceptance rate of the passenger requests is relatively small, e.g. SS and SC(South) decrease by about 5\%-10\%.

    \begin{figure}[h!]
       \centering
       \includegraphics[width=.7\textwidth]{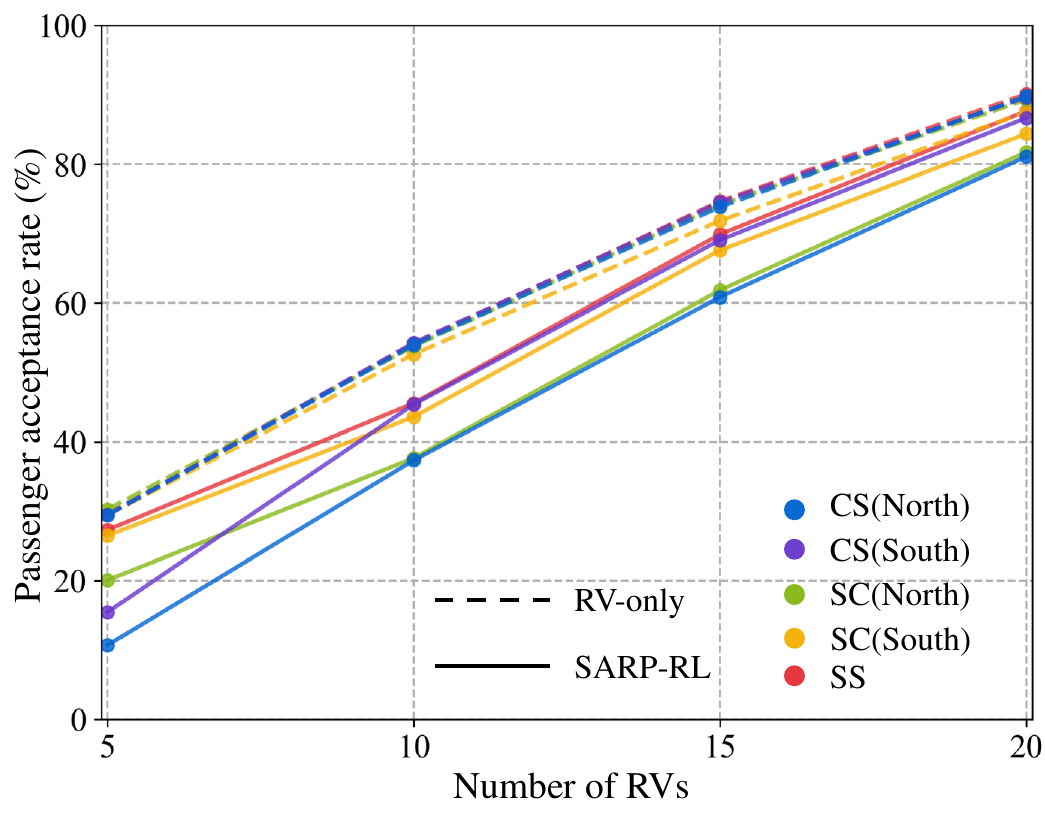}
       \caption{Passenger acceptance rate (\%) of five parcel distributions under the same number of RVs. The experimental results are based on the demand scenarios SS-76-24, SC(South)-76-24, SC(North)-76-24, CS(South)-76-24 and CS(North)-76-24.}
       \label{fig_acceptance_rate}
   \end{figure}

We explore the impact of different passenger/parcel request ratios when the total number of requests remains constant. Figure \ref{fig_different_requests} shows that when the proportion of parcel requests increases, the increase of total RV profits (\%) of the SARP-RL compared to RV-only increases. This is because more parcel requests can generate more additional profits for RVs. Figure \ref{SARP_different_parcel_num} shows that with the increase of the proportion of parcel requests, compared with SARP, SARP-RL has more obvious savings on logistics costs (LV fleet size) but also greater impact on the RV profits. The latter is because the profits of RVs from parcel requests are lower than that from passenger requests. To support this view, we conduct a sensitivity analysis on the unit price of parcel transport service (i.e. variable income per kilometer for serving a parcel request $\gamma_2$), as shown in Figure \ref{SARP_different_gamma2}. The results indicate that with the increase in $\gamma_2$, the impact on RV profits is reduced but the logistics cost (LV fleet size) that can be saved is also reduced for SARP-RL versus SARP.

    \begin{figure}[H]
       \centering
       \includegraphics[width=.7\textwidth]{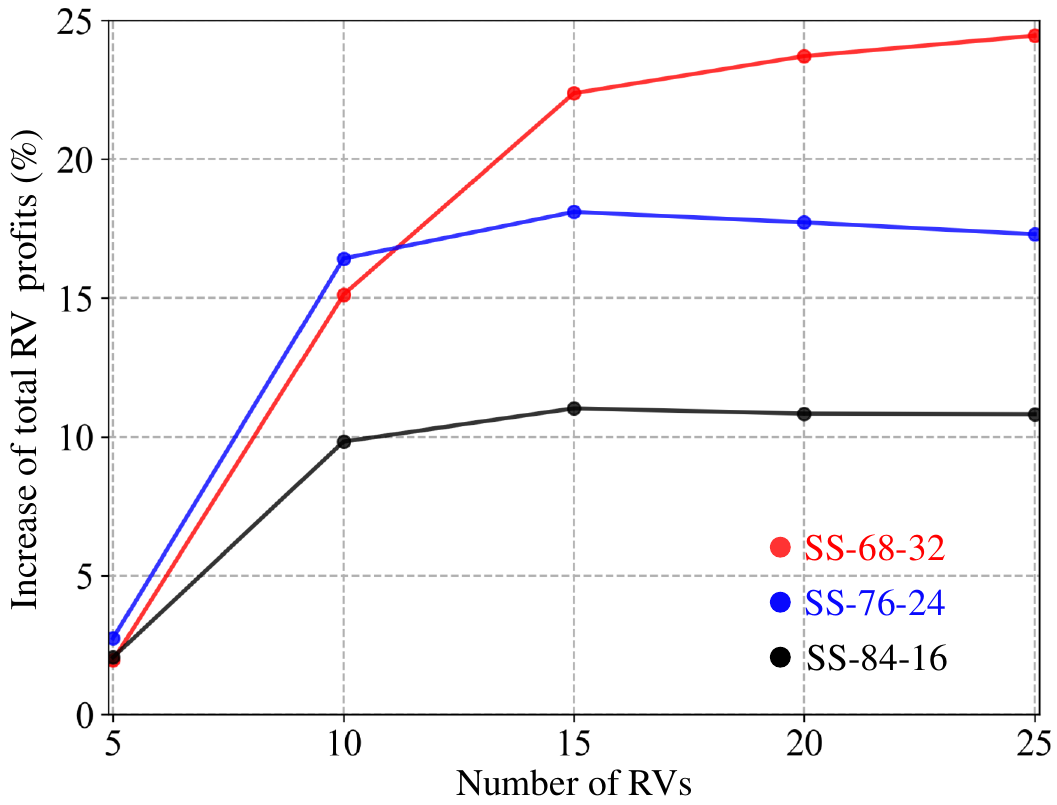}
       \caption{Increase of total RV profits (\%) of different passenger/parcel request ratios under the same number of RVs.}
       \label{fig_different_requests}
   \end{figure}

    \begin{figure}[H]
       \centering
       \includegraphics[width=1.0\textwidth]{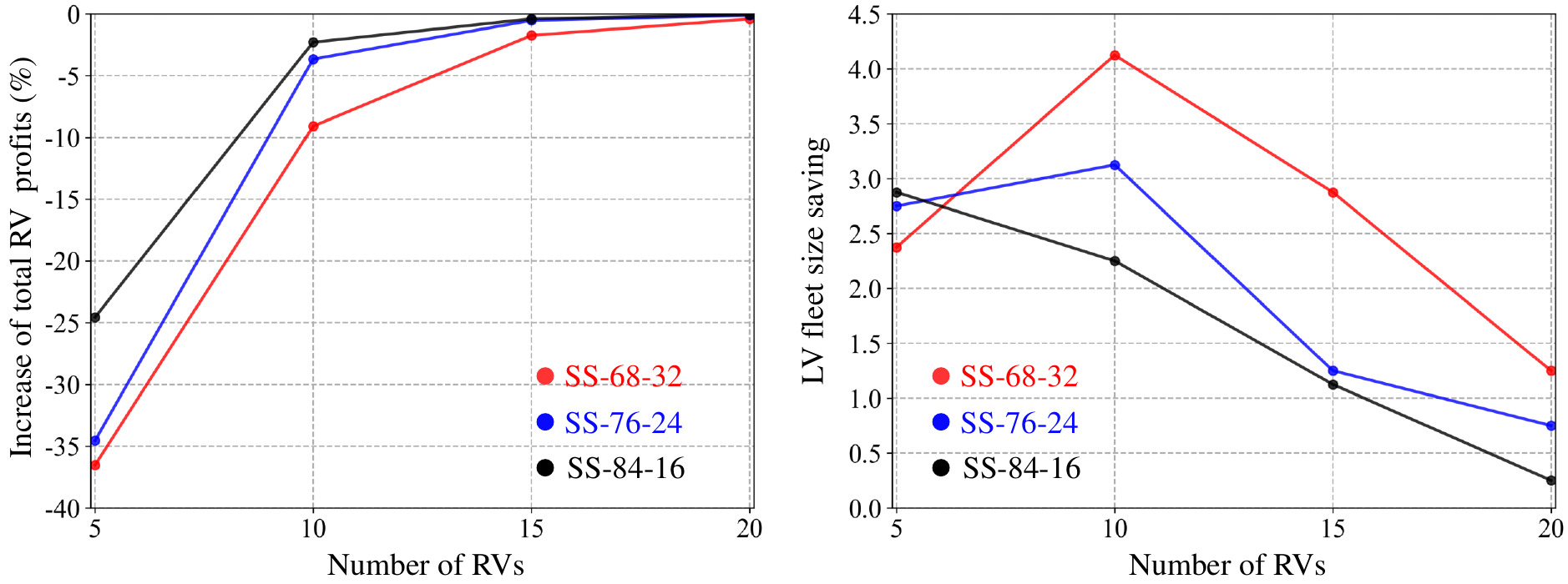}
       \caption{Left: Increase of total RV profits (\%) of different passenger/parcel request ratios under the same number of RVs for SARP-RL versus SARP. Right: LV fleet size saving of different passenger/parcel request ratios under the same number of RVs for SARP-RL versus SARP.}
       \label{SARP_different_parcel_num}
   \end{figure}

    \begin{figure}[H]
       \centering
       \includegraphics[width=1.0\textwidth]{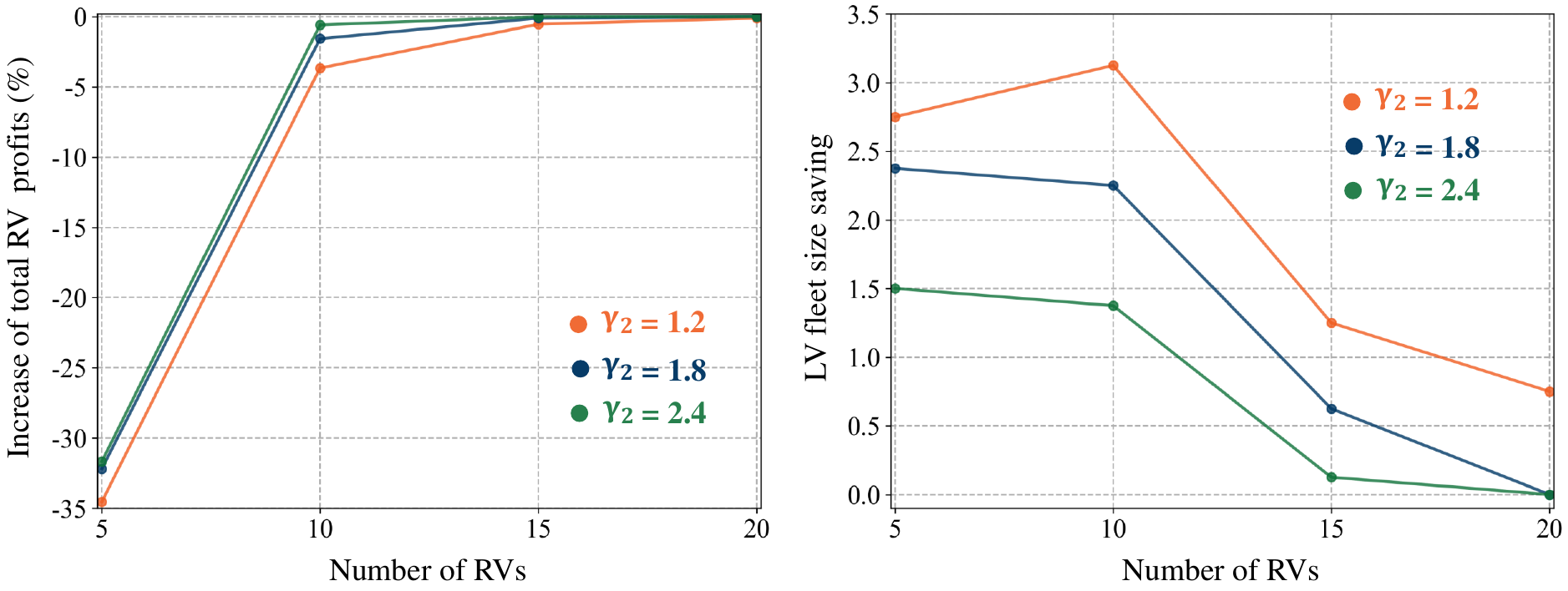}
       \caption{Left: Increase of total RV profits (\%) of different $\gamma_2$ under the same number of RVs for SARP-RL versus SARP. Right: LV fleet size saving of different $\gamma_2$ under the same number of RVs for SARP-RL versus SARP. The experimental results are based on the demand scenario SS-76-24.}
       \label{SARP_different_gamma2}
    \end{figure}

\subsection{Computational time}

The computational time of the entire solution procedure consists of three parts: 
\begin{itemize}
\item[(1)] The time to solve for the optimal route for given requests \eqref{SARP_1}-\eqref{SARP_18}, which is a MILP with limited size as the maximum trip length $l$ is quite small (see Figure \ref{fig_trips}). According to Table \ref{tab_trip_combinations}, the average time taken to perform this routine is $3\times 10^{-3}$ s. 

\item[(2)] The time to generate all possible trips. According to Table \ref{tab_trip_combinations}, the proposed solution scheme (with Algorithm \ref{alg_GNT}) requires solving the aforementioned MILP between $5.5\times 10^5$ and $9.7\times 10^5$ times, which amounts to 1650 s to 2910 s. It should be noted that such computations can be accelerated via parallel computing as these MILPs are completely independent. 

\item[(3)] The time to solve the trip-vehicle assignment problem (Algorithm \ref{alg_SF}). For example, it takes about 10 s to solve the $\varepsilon$-constraint method for $N_{L}^{\text{min}}=7$. 

\end{itemize}

In summary, the vast majority of the total solution time is associated with generating all feasible trips, which is considerably reduced by virtue of Algorithm \ref{alg_GNT} (by a factor of 7 to 11 compared to Algorithm \ref{alg_GAT}; see Table \ref{tab_trip_combinations}), and can be further reduced via straightforward parallelization.

\section{Conclusion and discussion} \label{sec_CD}
	This paper extends the share-a-ride problem (SARP), which optimizes ride-hailing vehicle (RV) operations to serve passenger and parcel requests, to the {\it share-a-ride problem with ride-hailing and logistic vehicles} (SARP-RL), which internalizes logistic operations to enable a balanced and systematic analysis of integrated passenger-freight transport based on shared mobility.

	The SARP-RL attempts to maximize the total RV profits and minimize the LV fleet size. To solve such a bi-objective optimization problem, we propose an exact solution method based on a decomposition scheme and an $\varepsilon$-constraint method  to compute all Pareto-optimal solutions, which can be solved efficiently using off-the-shelf solvers. The following findings are made from extensive numerical tests.

\begin{itemize}
    \item SARP-RL can reduce the number of LVs and increase RV driver profits compared to RV-only and LV-only. When the number of RVs increases, the proposed solution excels in terms of both LV fleet size (no need for any LVs) and RV profits (mostly over 15\% increase), compared to the RV-only and LV-only modes. 
    \item SARP-RL can reduce the fleet size of LVs compared to SARP via coordinated RV and LV operations. This implies that unilaterally maximizing RV profits when serving both passenger and parcel requests does not always reduce logistic costs. 
    
    \item SARP-RL may lead to a decrease in passenger acceptance rate compared to the RV-only mode. In particular, if the cluster point of the origin or destination of the parcel requests is located in remote areas, this has a significant impact on the acceptance rate of the passenger requests.
 
    \item When the proportion of parcel requests is higher, SARP-RL saves more on logistics costs (LV fleet size) than SARP (because the former explicitly considers logistic costs), but at the price of more RV profits.
    
    \item The computational time of the proposed solution framework is predominantly determined by the number of feasible trips. To enumerate all such trips and check their feasibility by solving MILPs, this work proposes a novel enumeration method that can reduce the number of MILPs to be solved by 7-11 times compared to existing literature \citep{ASWFR2017}. This enumeration procedure can be further accelerated via parallel computing.
\end{itemize}

\section*{Data availability}
The datasets that support the findings of this study are openly available in \url{https://github.com/Shenglin807/SARP_RL_Dataset/tree/master}.

\section*{Acknowledgement}
This work is partly supported by the National Natural Science Foundation of China through grants 72071163.

\end{document}